\documentclass[11pt,reqno]{amsart}
\usepackage{color}
\usepackage{graphics,epic}
\usepackage{amsmath,amssymb, amsthm}
\usepackage[all,2cell]{xy}
\usepackage{soul}
\usepackage[colorlinks=true,citecolor=red,linkcolor=blue]{hyperref}

\newtheorem{theorem}{Theorem}[section]
\newtheorem*{theorem*}{Theorem}

\newtheorem{lemma}[theorem]{Lemma}
\newtheorem{proposition}[theorem]{Proposition}
\newtheorem{corollary}[theorem]{Corollary}

\newtheorem*{conjecture*}{Conjecture}

\newtheorem{remark}[theorem]{Remark}
\newtheorem{definition}[theorem]{Definition}

\newcommand\cC{{\mathcal C}}

\newcommand{\cF}{{\mathcal F}}
\newcommand{\cG}{{\mathcal G}}

\renewcommand{\hat}[1]{\widehat{#1}}

\newcommand{\id}{{\rm id}}

\newcommand{\Hom}{{\rm Hom}\,}

\newcommand{\Res}{{\rm Res}\,}

\newcommand{\Z}{\mathbb{Z}}

\newcommand{\C}{\mathbb{C}}
\newcommand{\N}{\mathbb{N}}

\def\Res{{\rm Res}}
\def\wt{{\rm wt}}
\def\C{{\mathbb C}}

\def\Z{{\mathbb Z}}

\def\N{{\mathbb N}}

\def\1{{\bf 1}}

\def \Hom{{\rm Hom}}

\def \Ind{{\rm Ind}}
\def \pf{\noindent {\bf Proof: \,}}
\def\theequation{5.\arabic{equation}}

\def \h{\mathfrak{h}}

\def \g{\mathfrak{g}}

\begin{document}

\title[Tensor categories of VOAs $L_{-1}(sp(2n))$ and $M(1)^+$]{Tensor category of $\mathbb{Z}_2$-orbifold of Heisenberg vertex operator algebra and its applications}

\author{Dra\v{z}en Adamovi\'{c} }
\address{Dra\v{z}en Adamovi\'{c}, Department of Mathematics, Faculty of Science, University of Zagreb, Croatia}
\email{adamovic@math.hr}

\author{Xingjun Lin}
\address{Xingjun Lin,  School of Mathematics and Statistics, Wuhan University, Wuhan 430072,  China}
\email{linxingjun88@126.com}

\author{Jinwei Yang}
\address{Jinwei Yang, School of Mathematical Sciences, Shanghai Jiao Tong University, Shanghai 200240,  China.}
\email{jinwei2@sjtu.edu.cn}

\begin{abstract}
In this paper, we prove the category of finite length modules for the $\Z_2$-orbifold $M(1)^+$ of the Heisenberg vertex operator algebra whose simple composition factors are $M(1)^\pm$ or $M(1,\lambda)$ for $\lambda \in \C^\times$ is a vertex and braided tensor category. Our strategy is to show these simple composition factors are $C_1$-cofinite and the category of finite length $M(1)^+$-modules is exactly the category of grading-restricted $C_1$-cofinite modules. We also determine the fusion product decompositions of simple objects and prove the rigidity of this category. 

As an application of the tensor category structure of $M(1)^+$-modules, we prove the category $\cC_{-1}(sp(2n))$ of grading-restricted generalized modules for the simple affine vertex algebra $L_{-1}(sp(2n))$ is semisimple. For this, we first prove $M(1)^+$ and simple affine vertex algebra $L_{-\frac{1}{2}}(sp(2n))$ form a commutant pair in the simple minimal $W$-algebra $W_{-1}^{min}(sp(2n))$ for $n \geq 2$ and determine $W_{-1}^{min}(sp(2n))$ as well as its irreducible modules obtained from quantum Hamilton reduction as decompositions of $M(1)^+ \otimes L_{-\frac{1}{2}}(sp(2n))$-modules, then we show all the highest weight modules for $L_{-1}(sp(2n))$ in $\cC_{-1}(sp(2n))$ are irreducible via the quantum Hamilton reduction.

We also prove a Schur-Weyl duality between $L_{-1}(sp(2n))$ and $M(1)^+$ by showing they form a commutant pair in the $\Z_2$-orbifold of the rank $n$ $\beta\gamma$ system, and then establish a braided reversed equivalence between the category $\cC_{-1}(sp(2n))$ and the full subcategory of $C_1$-cofinite $M(1)^+$-modules consisting of direct sums of irreducible modules $M(1)^\pm$ and $M\big(1, \frac{s}{\sqrt{-2n}}\big)$ for $s \in \Z_{\geq 0}$. Consequently, we obtain the rigidity and fusion rules of the category $\cC_{-1}(sp(2n))$.
\end{abstract}
\maketitle

\tableofcontents
\section{Introduction }
\def\theequation{1.\arabic{equation}}
\setcounter{equation}{0}

\subsection{Tensor categories of $M(1)^+$-modules}
Tensor category of vertex operator algebras plays an important role in representation theory of vertex operator algebras. In the tensor category theory of vertex operator algebras, the first major difficulty is to establish the tensor category structure on certain representation categories of vertex operator algebras. When the vertex operator algebra $V$ is $C_2$-cofinite, Y.-Z. Huang proved that the category of grading restricted generalized modules is a vertex and braided tensor category (\cite{H3}), and if the vertex operator algebra is also rational, then the category is a modular tensor category (\cite{H2}). It is natural to ask whether there are tensor category structures on appropriate representation categories of non-$C_2$-cofinite vertex operator algebras. 

For non $C_2$-cofinite vertex operator algebras, it is important to study the category of $C_1$-cofinite modules because this category is closed under taking fusion product (\cite{Miy}) and the matrix elements of products and iterates of intertwining operators satisfy the differential equations (\cite{H1}). In fact, Huang recently proved the category of $C_1$-cofinite modules is a braided monoidal category (\cite{H4}). However, it is not clear whether the irreducible modules are $C_1$-cofinite, and whether the category of $C_1$-cofinite modules is an abelian category, particularly whether the category is closed under taking submodules and contragredient duals.

Previously in \cite{CY}, T. Creutzig and the third named author showed that if the category of $C_1$-cofinite $V$-modules is the same as the category of the finite length $V$-modules whose simple composition factors are irreducible grading-restricted $C_1$-cofinite modules, then the category of finite length modules is a vertex and braided tensor category. Using this idea, we have been able to establish the braided tensor categories for some non-$C_2$-cofinite vertex operator algebras including the Virasoro vertex operator algebra at all central charges (\cite{CJORY}), affine vertex superalgebra $\mathfrak{gl}(1|1)$ (\cite{CMY2}), affine Lie algebra $\mathfrak{sl}_2$ at positive rational levels (\cite{MY}), singlet algebras (\cite{CMY1}, \cite{CMY3}), $N=1$ super Virasoro algebra at all central charges (\cite{CMOY}) and $N=2$ super Virasoro algebra at central charge ${3k}/{(k+2)}$ for admissible $k$ of $sl(2)$ (\cite{C}).

In this paper, we construct and study the tensor categories arising from the $\Z_2$-orbifold $M_\h(1)^+$ of the Hisenberg vertex operator algebra $M_\h(1)$ generated by a one-dimensional Lie algebra $\mathfrak{h}$, introduced by C. Dong and R. Griess (\cite{DG}). We simply write this orbifold vertex operator algebra as $M(1)^+$ and Heisenberg vertex operator algebra as $M(1)$ if $\mathfrak{h}$ is clear. It is one of the most fundamental vertex operator algebras coming from orbifold theory, neither the Heisenberg vertex operator algebra nor this orbifold vertex operator algebra itself is rational or $C_2$-cofinite. Its simple modules were classified by Dong and K. Nagatomo (\cite{DN1}), they are: the fixed point and the eigenvalue $-1$-subspaces $M(1)^\pm$ by the involution $\theta$ of $M(1)$ induced by sending $h \mapsto -h$ for $h \in \h$, the Heisenberg simple modules $M(1,\lambda)$ for $\lambda \in \mathfrak{h}^* - \{0\} \cong \C^\times$,  and the fixed point and eigenvalue $-1$ subspaces of the Heisenberg twisted module $M(1)(\theta)^\pm$. The category of grading-restricted generalized $M(1)^+$-modules is not semisimple, Abe (\cite{Abe}) calculated extensions between irreducible
$M(1)^+$--modules and proved that 
the extension group   $\mbox{Ext} ^1 (L_1, L_2)  $ is trivial except  for the pairs
$(M(1)^+, M(1)^-), \;  (M(1)^-, M(1)^+), \;     (M(1, \lambda) , M(1, \lambda)).$

To describe the category of $C_1$--cofinite $M(1) ^+$--modules, we first need to analyze whether the irreducible  $M(1) ^+$-modules are $C_1$--cofinite. The strategy is similar to the proof of $C_1$-cofiniteness of irreducible modules for the singlet algebras (\cite{CMR}). We first show the atypical modules $M\big(1, \frac{n}{\sqrt{2}}\big)$ for $n \in \Z_{\geq 0}$ are $C_1$-cofinite by finding a spanning set of these modules regarding as modules for the simple Virasoro algebra $L(1,0)$ of central charge $1$ and using the fact the irreducible modules $L\big(1, \frac{n^2}{4}\big)$ for the Virasoro algebra $L(1,0)$ are $C_1$-cofinite. Then we show the modules $M(1,\lambda)$ are $C_1$-cofinite for all $\lambda \in \C$ but possibly only a countable set. Finally, we exclude this countable set by using the fact $C_1$-cofinite modules are closed under taking fusion products (\cite{Miy}). To summarize, we prove  (Theorem \ref{cofinite1} and \ref{cofinite2} in the main context):
\begin{theorem} The   $M(1) ^+$--modules
$$ M(1) ^{\pm},\; M(1,\lambda) (\lambda \ne 0),  $$
provide a complete list of  irreducible $C_1$--cofinite $M(1) ^+$--modules.
\end{theorem}
This list is smaller than the list of irreducible grading-restricted generalized $M(1)^+$--modules, since it excludes the modules $M(1)(\theta)^{\pm}$, which we prove are not $C_1$--cofinite. 

Now the simple modules $M(1)^\pm$ and $M(1, \lambda)$ are $C_1$-cofinite, all the finite length modules with simple composition factors $M(1)^\pm$ or $M(1, \lambda)$ are $C_1$-cofinite. We proceed to show that they actually exhaust all the $C_1$-cofinite modules by classifying all the highest weight $M(1)^+$--modules (see Theorem \ref{thm:irreduciblehwm} in the context):
\begin{theorem}
    Suppose $M$ is a highest weight module with the highest weight vector $v$, then
    \begin{itemize}
        \item[(1)]If the conformal weight of $v$ is nonzero, then $M$ is irreducible.
        \item[(2)]If the conformal weight of $v$ is zero, then either $M$ is isomorphic to $M(1)^+$ or a length $2$ logarithmic module with the following non-split exact sequence
        \[
        0 \rightarrow M(1)^- \rightarrow M \rightarrow M(1)^+ \rightarrow 0.
        \]
    \end{itemize}
\end{theorem}
Once we show the category of finite length modules with composition factors $M(1)^\pm$ or $M(1,\lambda)$ for $\lambda \in \C^\times$ is exactly the category $\cC_1(M(1)^+)$ of grading-restricted $C_1$-cofinite modules, it follows from \cite[Theorem~3.6]{CY} that the category of finite length modules is a vertex and braided tensor category.


We also determine the fusion product decompositions of simple objects in the category $\cC_1(M(1)^+)$. Previously, Abe computed the fusion rules among three irreducible modules (\cite{Abe3}), but since the category is not semisimple, the fusion rules cannot quite determine the fusion product decomposition completely. To solve this problem, we use the tensor category for vertex operator algebra extensions developed in \cite{HKL} and \cite{CKM} by regarding the Heisenberg vertex operator algebra $M(1)$ as an algebra object in the category $\cC_1(M(1)^+)$ and using fusion product decompositions of $M(1)$-modules. The fusion product decompositions of simple objects are as follows (see Theorem \ref{fr-1}):
\begin{theorem} 
For $\lambda, \mu \in \C^\times$,
\begin{align*}
&M(1)^- \boxtimes M(1)^- = M(1)^+,\\
&M(1)^- \boxtimes M(1, \lambda)= M(1, \lambda),\\
& M(1, \lambda) \boxtimes M(1, \mu) = M(1, \lambda+\mu) \oplus M(1, \lambda - \mu).
\end{align*}
\end{theorem}

From the fusion product decomposition, it is obvious that $M(1)^-$ is a simple current and the simple objects $M(1,\lambda)$ are non-negligible with moderate growth. Then it follows from \cite[Theorem~1.1]{EP} by P. Etingof and D. Penneys that $M(1,\lambda)$ are rigid. Therefore, the category $\cC_1(M(1)^+)$ is rigid by \cite[Theorem~4.4.1]{CMY1}.

\subsection{Semisimplicity of $KL_{-1}(sp(2n))$} 
Using the tensor category structure on the category $\cC_1(M(1)^+)$, we prove the category $\cC_{-1}(sp(2n))$ of grading-restricted generalized modules for the simple affine vertex operator algebra $L_{-1}(sp(2n))$ is semisimple. 

One important goal of the theory of affine vertex algebra is to construct semisimple categories of representations and provide them with the tensor category structure. Let $V^k(\g)$ be the universal affine vertex algebra associated to the Lie algebra $\g$, and let $L_k(\g)$ be its simple quotient, they both have vertex operator algebra structures when $k \neq -h^\vee$ (\cite{FZ}). There are two natural categories of $L_k(\g)$-modules to study, one is the category $\cC_k(\g)$ of grading-restricted generalized $L_k(\g)$-modules, the other is the Kazhdan-Lusztig category $KL_k(\g)$ consisting of finite length modules with simple composition factors in $\cC_k(\g)$. A basic question is whether $KL_k(\g)$ is a braided tensor category. 

When $k+h^\vee \notin \mathbb{Q}_{\geq 0}$, D. Kazhdan and G. Lusztig proved that the category $KL_k(\g)$ is a rigid braided tensor category (\cite{KL}). When $k$ is admissible, T. Arakawa proved the (bigger) category $\mathcal O_k$ is semisimple (\cite{A2}) and hence $KL_k(\g)$ is also semisimple, therefore $KL_k(\g)$ is also a braided tensor category (\cite{CHY}). It remains to study the case when $k$ is beyond admissible. 

One breakthrough on the representations beyond admissible level was made by the first named author with his collaborators in \cite{AKMPP2}, where it was proved that the subcategory of $L_{k}(\g)$-modules that are locally finite as $\g$-modules with semisimple $L(0)$-actions and $L(0)$-weights are bounded from below (they call this category $KL_k^{fin}(\g)$ in \cite{AKMPP2}) is semisimple if one of the following conditions holds:
\begin{itemize}
\item $k$ is a collapsing level for minimal affine $W$-algebra $W_k ^{min}(\g)$;
\item $W_k ^{min}(\g)$ is rational;
\item the category of ordinary $W_k ^{min}(\g)$--modules is semisimple. By ordinary modules we mean grading-restricted generalized modules with semisimple $L(0)$-actions.
\end{itemize}
Using this result, Creutzig and the third named author were able to prove the category $\cC_k(\g)$ is semisimple, and hence $KL_k(\g) = \cC_k(\g)$ have braided tensor category structures for $\g, k$ satisfying the above conditions (\cite{CY}). 

Some new cases in which $KL_k(\mathfrak{g})$ is semisimple were discovered in \cite{AALY, ACPV}. 
In \cite{ACPV} it was shown that $KL_k(\mathfrak{g})$ may be semisimple even when the category of ordinary 
$W_k ^{min}(\g)$--modules is not semisimple. 
In such situations the analysis is necessarily more subtle, as it requires detailed structural results 
on the category of $W_k ^{min}(\g)$--modules. 

In this paper, we study the category $KL_k(\g)$ for $\g = sp(2n)$ and $k = -1$. The category of ordinary $W_k ^{min}(\g)$--modules is indeed not semisimple. To establish semisimplicity of $\cC_{-1}(sp(2n))$, we apply \cite[Theorem~4.9]{CY} (cf. \cite{AKMPP2}), which states:
\begin{itemize}
    \item[(C)] If every highest weight $L_k(\g)$-module in $\cC_k(\g)$ is irreducible, then the category $\cC_k(\g)$ is semisimple.
\end{itemize}
Verifying condition (C) directly using the representation theory of $L_k(\g)$ is difficult. Nevertheless, it was observed in \cite{AKMPP2} that condition (C) follows from the weaker criterion:
\begin{itemize}
    \item[(C')] If, for every highest weight $L_k(\g)$-module $M$ in $KL_k(\g)$, the $W_k ^{min}(\g)$-module $H_{\theta}(M)$ is irreducible, then $KL_k(\g)$ is semisimple. Here $H_{\theta}$ denotes the quantum Hamiltonian reduction functor.
\end{itemize}
Therefore, our strategy is to verify condition (C'). The  statement of (C') is for $KL_k(\g)$, we actually show it also holds for the bigger category $\cC_k(\g)$. For this, we need to use the tensor category structure on $\mathcal{C}_1(M(1)^+)$ to show that no non-split extensions appear in the decomposition of the $W_{-1}^{min}(sp(2n))$-module $H_{\theta}(M)$, where $M$ is an $L_{-1}(sp(2n))$-module in $\mathcal{C}_{-1}(sp(2n))$. This establishes condition (C') and completes the proof of semisimplicity of $\mathcal{C}_{-1}(sp(2n))$.

As a byproduct, we also solved two problems that are of interest in the representation theory of vertex operator algebras. Firstly, we give a complete description of the structure of the minimal affine $W$-algebra $W_k^{min}(sp(2n))$ at the non-collapsing, non-conformal level $k=-1$. We prove the vertex operator algebras $L_{-\frac{1}{2}}(sp(2n-2))$ and $M(1)^+$ form a commutant pair in $W_{-1}^{min}(sp(2n))$ and as an $L_{-\frac{1}{2}}(sp(2n-2)) \otimes M(1)^+$-module,
$$ W_{-1}^{min}(sp(2n)) \cong L_{sp(2n-2)}\left(-\frac{1}{2},0\right) \otimes M(1)^+ \oplus  L_{sp(2n-2)}\left(-\frac{1}{2},\lambda_1\right) \otimes M(1)^-, $$   
where $\lambda_1$ is the first fundamental weight of $sp(2n-2)$ (see Theorem \ref{structureA}). The decomposition of minimal affine $W$-algebras as modules over their affine vertex subalgebras is largely unknown. While it has been studied in the case of conformal embeddings, the non-conformal setting remains considerably more difficult. 

Secondly, we extend the classification of irreducible modules for $L_{-1}(sp(2n))$, previously established for $n \ge 3$, to the remaining case $n=2$. The previous work \cite{AP2} provides a classification of irreducible modules in $KL_{-1}(sp(2n))$ for $n \ge 3$, based on the existence of a determinant-type singular vector of conformal weight three in  the universal affine vertex algebra $V^{-1}(sp(2n))$. The case $n=2$ remained open due to the complexity of the corresponding singular vectors. In the present paper, we resolve this case using a $W$-algebra approach, which avoids the need for explicit formulas for singular vectors (see Theorem~\ref{clasirr}).

\subsection{Schur-Weyl duality between $L_{-1}(sp(2n))$ and $M(1)^+$}
We study the tensor category structure of $KL_{-1}(sp(2n))$ by proving a Schur-Weyl duality between $L_{-1}(sp(2n))$ and $M(1)^+$ (see Theorem \ref{dec-mn} in the main context):
\begin{theorem}
The $\Z_2$-orbifold subalgebra $\mathcal{S}(2n)^+$ of the Weyl vertex algebra $\mathcal{S}(2n)$ viewed as an $L_{-1}(sp(2n))\otimes M(1)^+$-module has the following decomposition:
\begin{align*}
\mathcal{S}(2n)^+ \cong &  L_{-1}(sp(2n))\otimes M(1)^+\oplus L_{sp(2n)}(-1, \bar{\Lambda}_2)\otimes M(1)^-\\
  &\quad \oplus   \bigoplus_{s = 1}^\infty L_{sp(2n)}(-1, s\bar{\Lambda}_1) \otimes M\left(1,\frac{s}{\sqrt{-2n}}\right),
\end{align*}
where $\bar{\Lambda}_1$ and $\bar{\Lambda}_2$ are the first two fundamental weights of $sp(2n)$.
 \end{theorem}

Using this theorem and \cite[Theorem~1.1]{M3} by R. McRae, 
\begin{theorem}
We have the following results:\\
(1) The full subcategory $\cC_1^{\Z}(M(1)^+)$ consisting of direct sums of $M(1)^\pm$ and $M\big(1, \frac{s}{\sqrt{-2n}}\big)$ for $s \in \Z_{\geq 1}$ form a tensor subcategory of $\cC_1(M(1)^+)$.\\
(2) The category $KL_{-1}(sp(2n))$ is braided reversed equivalent to the category $\cC_1^{\Z}(M(1)^+)$. In particular, $KL_{-1}(sp(2n))$ is rigid with the following fusion product decompsotion:
         \begin{align*}
&        L_{sp(2n)}(-1, \bar{\Lambda}_2) \boxtimes L_{sp(2n)}(-1, \bar{\Lambda}_2) \cong L_{-1}(sp(2n)),\\
&  L_{sp(2n)}(-1, \bar{\Lambda}_2) \boxtimes L_{sp(2n)}(-1, s\bar{\Lambda}_1) \cong L_{sp(2n}(-1, s\bar{\Lambda}_1),\\
& L_{sp(2n)}(-1, s\bar{\Lambda}_1) \boxtimes L_{sp(2n)}(-1, t\bar{\Lambda}_1) \cong L_{sp(2n)}(-1, (s+t)\bar{\Lambda}_1) \oplus L_{sp(2n)}(-1, (s-t)\bar{\Lambda}_1)
    \end{align*}
for $s, t \in \Z_{\geq 1}$ and $s \geq t$.
\end{theorem}

\section{Preliminaries}
\subsection{$\mathbb{Z}_2$-orbifold of Heisenberg vertex operator algebra and its representations}

In this subsection, we recall some basic facts about the $\mathbb{Z}_2$-orbifold of Heisenberg vertex operator algebras from \cite{FLM}, \cite{DN1}, \cite{DN2}, \cite{Abe}. 

Let $\h$ be a finite
dimensional vector space of dimension $d$ with a non-degenerate symmetric bilinear
form $(\cdot,\cdot)$. Define the affine Lie algebra
$$\hat{\h}=\h\otimes \C[t, t^{-1}]\oplus \C K$$ with Lie brackets: 
$$[\alpha(m), \beta(n)]=m(\alpha, \beta)\delta_{m+n,
0}K,\;\;\; [\hat{\h}, K]=0$$
for $\alpha, \beta\in \h$, $m,n\in \mathbb{Z}$, and $\alpha(n)=\alpha\otimes t^n$.

For any $\lambda\in \h$, set $$M_\h(1,
\lambda)=U(\hat{\h})/J_{\lambda},$$ where $J_{\lambda}$ is the left
ideal of $U(\hat{\h})$ generated by $\alpha(n),(n\geq 1)$,
$\alpha(0)-(\alpha, \lambda)$ and $K-1$. Set
$e^{\lambda}=1+J_{\lambda}$. Then $M_\h(1,\lambda)$ is
spanned by $\alpha_1(-n_1)\cdots \alpha_k(-n_k)e^{\lambda}$,
$n_1\geq \cdots \geq n_k\geq 1$. It is known that $M_\h(1) : = M_\h(1,0)$
is a vertex operator algebra such that $\1=1+J_0$ is the vacuum vector
and 
 \begin{equation}\label{conformalvh}
      \omega=\frac{1}{2}\sum_{1\leq
i\leq d}h_i(-1)^21,
 \end{equation}
 where $\{h_1,\cdots, h_d\}$
is an orthonormal basis of $\h$,
is a conformal element. Furthermore, $M_\h(1,
\lambda)$ is an irreducible $M_\h(1)$-module.

Let $\theta$ be an automorphism of $M_\h(1)$ defined by
$$\theta(\alpha_{1}(-n_{1})\alpha_{2}(-n_{2})\cdots \alpha_{k}(-n_{k})\1)=
(-1)^{k}\alpha_{1}(-n_{1})\alpha_{2}(-n_{2})\cdots
\alpha_{k}(-n_{k})\1$$
for $\alpha_{i}\in \h,\ n_{i}\in \Z_{>0}$. Denote the $\theta$-fixed point subalgebra of
$M_\h(1)$ by $M_\h(1)^+$ and the $\theta$-eigenspace of $M_\h(1)$ with eigenvalue $-1$ by $M_\h(1)^-$. Then the vertex operator algebra $M_\h(1)^+$ is simple, and $M_\h(1)^-$, $M_\h(1,\lambda)$ $(\lambda\neq0)$ are irreducible $M_\h(1)^+$-modules. Moreover,
$M_\h(1,\lambda)$  is isomorphic to $M_\h(1, -\lambda)$  as $M_\h(1)^+$-modules (\cite[Proposition~2.2.2]{DN2}).

Next, we introduce the construction of the $\theta$-twisted modules for $M_\h(1)$. Let $$\hat{\h}[-1]= \h\otimes t^{\frac{1}{2}}\C[t,t^{-1}]\oplus\C K$$ be the Lie algebra with commutation relations $$[\alpha\otimes t^{m},
\beta\otimes t^{n}]=m\delta_{m+n,0} (\alpha, \beta) K,\;\;\; [K,\hat{\h}[-1]]=0$$ for $\alpha, \beta \in \h,\ m,n\in
1/2+\Z$. Set $\hat{\h}[-1]^+= \h\otimes t^{\frac{1}{2}}\C[t]\oplus\C K$, and denote by $\C = \C{\bf 1}_{tw}$ the one-dimensional $\hat{\h}[-1]^+$-module such that $\h\otimes t^{\frac{1}{2}}\C[t]$ acts as $0$ and $K$ acts as $\id_{\C}$. Define the induced module $M_\h(1)( \theta)$ of $\hat{\h}[-1]$ as follows: $$M_\h(1)( \theta): =\Ind_{\hat{\h}[-1]^+}^{\hat{\h}[-1]}\C.$$ By Poincar\' e-Birkhoff-Witt theorem,
$M_\h(1)( \theta)\cong S(\h\otimes t^{-\frac{1}{2}}\C[t^{-1}])$ as vector spaces.
Denote the action of $\alpha\otimes t^{n}$ $(\alpha\in\h,\ n\in {1/2}+\Z)$ on
$M_\h(1)( \theta)$ by $\alpha(n)$, and define the action of $\theta$ on $M_\h(1)( \theta)$ by
$$\theta(\alpha_{1}(-n_{1})\alpha_{2}(-n_{2})\cdots \alpha_{k}(-n_{k})\1_{tw})=
(-1)^{k}\alpha_{1}(-n_{1})\alpha_{2}(-n_{2})\cdots
\alpha_{k}(-n_{k})\1_{tw}$$
for $\alpha_{i}\in \h,\ n_{i}\in 1/2+\N$. Let $M_\h(1)( \theta)^{\pm}$ be the $\theta$-eigenspace of $M_\h(1)( \theta)$ with eigenvalues $\pm
1$, respectively. Then in \cite[Theorem~2.2.3]{DN2} it is proved that $M_\h(1)(\theta)^{\pm}$ are irreducible modules for $M_\h(1)^+$. Moreover, the following result has been established in \cite[Theorem~6.2.2]{DN2}.
\begin{theorem}
The set $\{M_\h(1)^{\pm}, M_\h(1,\lambda)(\lambda\neq 0),  M_\h(1)( \theta)^{\pm}\}$ provides a complete list of irreducible $M_\h(1)^+$-modules.
\end{theorem}

From now on, we assume $\h=\C \alpha$ such that $(\alpha, \alpha)=1$. For simplicity, we drop the subscript $\h$ in the vertex algebra $M_\h(1)^+$ and its modules $M_\h(1,\lambda)$, $M_\h(1)^{\pm}$ and $M_\h(1)( \theta)^{\pm}$. 

The vertex algebra $M(1)^+$ is generated by $\omega$ and a weight $4$ Virasoro algebra singular vector
   $$J=\alpha(-1)^4\1-2\alpha(-3)\alpha(-1)\1+\frac{3}{2}\alpha(-2)^2\1$$
(see \cite[Theorem~2.7]{DG}). The following table from \cite[Theorem~4.5]{DN1} gives the actions of the $\omega$ and $J$ on the top of simple $M(1)^+$-modules:
\begin{table}[h]
\begin{center}
\caption{}
\begin{tabular}{|c|c|c|c|c|c|c|}
\hline
$W$  &  $M(1)^+$ & $M(1)^-$ & $M(1,\lambda)$ & $M(1,\theta)^+$ & $M(1,\theta)^-$\\
\hline
$W_{\rm top}$ & $\mathbb{C}{\bf 1}$  & $\C h(-1){\bf 1}$  & $\mathbb{C}$ & $\mathbb{C}$ & $\C h(-1/2){\bf 1}$\\
\hline
$\omega$ & $0$ & $1$ & $\lambda^2/2$ & $1/16$ & $9/16$\\
\hline
$J$ & $0$ & $-6$ & $\lambda^4 - \lambda^2/2$ & $3/128$ & $-45/128$\\
\hline
\end{tabular}
\end{center}
\end{table}

The extensions of irreducible modules have been established in \cite[Theorem~5.5, 5.6]{Abe}:
\begin{theorem} \label{Abe}
Assume that  $(L_1, L_2)$  is an ordered pair of irreducible $M(1)^+$-modules.
The extension group   $\mbox{Ext} ^1 (L_1, L_2)  $ is non-trivial for pairs
$$   (M(1)^+, M(1)^-),   (M(1)^-, M(1)^+),      (M(1,\lambda) , M(1,\lambda)).$$
For all other pairs $(L_1, L_2)$  of irreducible $M(1)^+$-modules, the extension group is trivial. In particular, there is no nontrivial extensions between irreducible modules $\{M(1)(\theta)^{\pm}\}$ with any other irreducible modules.
\end{theorem}

\begin{remark}
We say $\{M(1)^{\pm}, M(1,\lambda)(\lambda\neq0)\}$ are $M(1)^+$-modules of untwisted type, and $\{M(1)( \theta)^{\pm}\}$ are $M(1)^+$-modules of twisted type.
\end{remark}

\subsection{Basics on affine vertex operator algebras and minimal $W$-algebras}\label{3-1}
In this subsection, we recall from \cite{FZ}, \cite{LL} some facts on vertex operator algebras associated to affine Lie algebras. Let $\g$ be a simple Lie algebra, $\h$ be a Cartan subalgebra  of $\g$. Choose a minimal root $-\theta$ of $\g$, and let $(\cdot|\cdot)$ be the normalized  non-degenerate invariant symmetric bilinear form of $\g$ such that $(\theta|\theta)=2$. Here we abuse the letter $\theta$ to denote both the involution of the Heisenberg vertex operator algebra and the minimal root, it should be clear when they appear in the context. The affine Lie algebra associated to $\g$ is defined on $\hat{\g}=\g\otimes \C[t^{-1}, t]\oplus \C K$ with Lie brackets
\begin{align*}
[x(m), y(n)]&=[x, y](m+n)+(x|y) m\delta_{m+n,0}K,\\
[K, \hat\g]&=0,
\end{align*}
for $x, y\in \g$ and $m,n \in \Z$, where $x(n)$ denotes $x\otimes t^n$.

Let $M$ be a $\g$-module and $k$ be a complex number, regarding $M$ as a $\g\otimes \C[t]\oplus \C K$-module such that $\g\otimes t\C[t]$ acts as $0$ and $K$ acts as $k\cdot \id_{M}$. Then we consider the induced module
\begin{align*}
\hat{M}_{k}=\Ind_{\g\otimes \C[t]\oplus \C K}^{\hat \g}M.
\end{align*}
Let $ \h^*$ be the dual space of   $\h$. For $\lambda\in \h^*$, we use $L_{\g}(\lambda)$ or just $L(\lambda)$ if $\g$ is clear, to denote the irreducible highest weight $\g$-module of highest weight $\lambda$. The induced $\hat{\g}$-module $\hat{L(\lambda)}_k$ is called the {\em generalized Verma module}, which we denote by $V_{\g}(k,\lambda)$, and we denote the irreducible quotient of $V_{\g}(k,\lambda)$ by $L_{\g}(k,\lambda)$. If $k \neq -h^\vee$, both the modules $V_{\g}(k,0)$ and $L_{\g}(k,0)$ have vertex operator algebra structure (\cite{FZ}). For simplicity, we denote them by $V^{k}(\g)$ and $L_{k}(\g)$, respectively.


\begin{definition}\label{KL}
We will mainly study the following two categories of $L_{k}(\g)$-modules:
\begin{itemize}
\item[(1)] Let $\cC_k(\g)$ be the category of grading-restricted generalized  $L_{k}(\g)$-modules.
\item[(2)] Let $KL_k(\g)$ be the full subcategory of finite length generalized $L_{k}(\g)$-modules whose simple composition factors are in $\mathcal{C}_k(\g)$.
\end{itemize}
\end{definition}

The following result has been established in \cite[Theorem~4.9]{CY} (cf. \cite[Theorem~4.3]{AMP}), which will play an important role in this paper.
\begin{theorem}\label{keysemi}
Let $\g$ be a simple Lie algebra. If all the generalized Verma modules of $L_{k}(\g)$ in $\cC_k(\g)$ are irreducible, then the category $\cC_k(\g)$ is semisimple.
\end{theorem}

If $\cC_k(\g)$ is semisimple, we can show that the category $KL_k(\g)$ is exactly the category of $C_1$-cofinite modules, and therefore has a vertex and braided tensor category structure. In fact, more generally, we have
\begin{theorem}[\cite{CY} Theorem~3.10]\label{tensorCY}
Let $\g$ be a finite dimensional simple Lie algebra. If all the generalized Verma modules of $L_{k}(\g)$ in $\cC_k(\g)$ are of finite length, then $KL_k(\g)$ has a braided tensor category structure.
\end{theorem}

A universal $W$-algebra is a vertex algebra $\mathcal{W}^k(\mathfrak{g}, f)$ associated to the triple $(\mathfrak{g}, f, k)$, where $\mathfrak{g}$ is a simple Lie algebra, $f$ is a nilpotent element of $\mathfrak{g}$, and $k \in \C$, by applying the quantum Hamiltonian reduction functor $H_f$ to a complex associated to the affine vertex algebra $V^k(\mathfrak{g})$ \cite{KRW, KW2}. In particular, it was shown that, for $k$ non-critical and $f$ part of an $\mathfrak{sl}_2$-triple, $\mathcal{W}^k(\mathfrak{g}, f)$ has a conformal vertex algebra structure with a conformal vector $\omega$. For $k$ non-critical the vertex algebra $\mathcal{W}^k(\mathfrak{g}, f)$ has a unique simple quotient, denoted by $\mathcal{W}_k(\mathfrak{g}, f)$.

For the minimal root $-\theta$, we may choose root vectors $e_{\theta}$ and $e_{-\theta}$ such that
\[
[e_{\theta}, e_{-\theta}] = x \in \mathfrak{h},\;\;\; [x, e_{\pm \theta}] = \pm e_{\pm \theta}.
\]
It is proved in \cite{KW2} that the universal minimal $W$-algebras of level $k$ corresponding to $f = e_{-\theta}$,  which we denote by $W^k _{min}(\g)$,
has a unique simple quotient, which we denote by $W_k^{min} (\g)$.

Let $H_{\theta}$ be the Hamilton reduction functor from the category $\mathcal{O}^k$ consisting of $\widehat \g$-modules of level $k$ that are $\widehat\h$-diagonalizable with finite dimensional weight spaces and finitely many maximal weights to the category of  $W_k^{min} (\g)$-modules. The following properties of $H_{\theta}$ were proved in \cite{A1}:
\begin{theorem}\label{exact}
We have
\begin{enumerate}
\item $H_{\theta}$ is exact.
\item $H_{\theta}$ sends ordinary $L_k(\mathfrak{g})$-modules to ordinary  $W_k^{min} (\g)$-modules.
\end{enumerate}
\end{theorem}

\section{The category of $C_1$-cofinite modules for $M(1)^+$}
We want to construct a braided tensor category structure on an appropriate category of $M(1)^+$-modules. In \cite{H4}, Y.-Z. Huang proved the category of $C_1$-cofinite grading-restricted generalized modules for a vertex operator algebra has a vertex and braided monoidal category structure, although not necessarily an abelain category. It still remains to give a complete description of the category of grading-restricted $C_1$-cofinite modules.

In this section, we first prove that the irreducible modules $M(1)^\pm$ and $M(1,\lambda)$ for $\lambda \in \C^\times$ are $C_1$-cofinite using the approach of \cite{CMR} and that the irreducible modules $M(1)(\theta)^\pm$ are not $C_1$-cofinite, and then determine the category of grading-restricted $C_1$-cofinite modules for $M(1)^+$ by identifying it with the category of finite length modules whose simple composition factors are $M(1)^\pm$ and $M(1,\lambda)$ for $\lambda \in \C^\times$. Therefore, the category of grading-restricted $C_1$-cofinite modules is a vertex and braided tensor category. The idea of establishing the braided tensor category is the same as the one used in the previous work such as \cite{CY} and \cite{CJORY}, where the authors establish the vertex and braided tensor category structure for certain affine and Virasoro vertex operator algebras.

\subsection{A spanning set of $M(1)^+$-modules as Virasoro modules}\label{sec:span} 
The subalgebra of $M(1)^+$ generated by the conformal element $\omega = \frac{1}{2}\alpha(-1)^2$ is isomorphic to the simple Virasoro vertex operator algebra of central charge $1$, which we denote by $L(1,0)$. In this subsection, we will give a spanning set for the irreducible $M(1)^+$-modules as $L(1,0)$-modules. 

Denote by $L(1,h)$ the unique irreducible highest weight $L(1,0)$-module with highest weight $h \in \C$. The $M(1)^+$-module $M(1,\lambda)$ regarded as an $L(1,0)$-module is irreducible if and only if $\lambda \neq \frac{m}{\sqrt{2}}$ for $m \in \Z$ (see for example \cite{KR}), and the $M(1)^+$-module $M\left(1, \frac{m}{\sqrt{2}}\right)$ regarded as an $L(1,0)$-module has the following decomposition:
 \begin{align}\label{decompositionatypical}
 M\left(1, \frac{m}{\sqrt{2}}\right) = \bigoplus_{k=0} ^{\infty}   L\left(1, \left(\frac{m}{2}+k\right)^2\right)  = \bigoplus_{k=0} ^{\infty}   L(1,0 ) v_m ^{(k)},
 \end{align}
where $v_m ^{(k)}$ is a highest weight vector of $L\left(1, \left(\frac{m}{2}+k\right)^2\right)$.

We first prove some operators induced from the element $J$ send the highest weight vector $v_m^{(k)}$ to $v_m^{(k+2)}$ plus the lower terms:
\begin{lemma}\label{keylemma}
 For each $k \in \Z_{\geq 0}$, there exists a nonzero constant $C$ such that
 \begin{equation}\label{eq:reduction}
 J_{-2m-4k-1}v^{(k)}_m = Cv^{(k+2)}_m + w,
 \end{equation}
 where $w \in  \bigoplus_{i= 0}^{k+1} L\left(1, \left(\frac{m}{2}+i\right)^2\right)$.
 \end{lemma}

To focus on the main streamline, we put the proof of Lemma \ref{keylemma} in the Appendix \ref{AppA}. Using Lemma \ref{keylemma}, we prove $M\big(1, \frac{m}{\sqrt{2}}\big)$ has the following spanning set:
 \begin{theorem}\label{lem:spanset}
For any $m\in \Z_{\geq 0}$, $M\big(1, \frac{m}{\sqrt{2}}\big)$ is spanned by the elements
 \begin{equation}\label{spanset}
 J_{-i_1}\cdots J_{-i_k}L_{-j_1}\cdots L_{-j_\ell}v_m^{(t)},
 \end{equation}
 where $i_1, \dots i_k, j_1, \dots, j_\ell \geq 1$, $k,\ell \in \mathbb{Z}_{\geq 0}$ and $t = 0, 1$.
 \end{theorem}
 \begin{proof}
Set $V_k = \bigoplus_{i= 0}^{k} L\left(1, \left(\frac{m}{2}+i\right)^2\right)$. We prove that $V_k$ is spanned by the element of the form \eqref{spanset} by induction on $k$. It is obvious when $k = 0, 1$. Now we assume the statement holds for $0, 1, \dots, k$. First it follows from \eqref{eq:reduction} that $v_m^{(k+1)}$ is also spanned by elements of the form \eqref{spanset}. Then an element of $V_{k+1}$ is of the form
\[
L_{-k_1}\cdots L_{-k_p}J_{-i_1}\cdots J_{-i_k}L_{-j_1}\cdots L_{-j_\ell}v^{(t)}_m,
\]
which is also of the form \eqref{spanset} because of the commutator formula
\[
[L_m, J_n] = (3(m+1) - n)J_{m+n}
\]
for $m,n \in \mathbb{Z}$.
 \end{proof}

\subsection{$C_1$-cofiniteness of irreducible $M(1)^+$-modules of untwisted type}
Recall that for a module $M$ for a vertex operator algebra $V$ is {\em $C_1$-cofinite} if the subspace
$$C_1(M)=\mbox{Span} \{v_{(-1)}m|v\in V, \wt~ v>0, m\in M \}$$ is of finite codimension in $M$.
  In this subsection, we show that the irreducible modules $M(1,\lambda)$ for $\lambda \in \C^\times$ and $M(1)^\pm$ are $C_1$-cofinite $M(1)^+$-modules.

We first show 
\begin{theorem}\label{cofinite1}
For any $m\in \Z_{\geq 0}$, the irreducible $M(1)^+$-module $M\big(1, \frac{m}{\sqrt{2}}\big)$ is $C_1$-cofinite.
\end{theorem}
\begin{proof}
The spanning set \eqref{spanset} given in Theorem \ref{lem:spanset} shows that
\[
M\left(1, \frac{m}{\sqrt{2}}\right) = C_1\left(M\left(1, \frac{m}{\sqrt{2}}\right)\right) + \bigoplus_{t = 0, 1}\bigoplus_{i=0}^{k_t} L_{-1}^{i}v_m^{(t)},
\]
where $k_t \in \Z_{\geq 0}$ depending only on the Virasoro algebra modules $L\big(1, (\frac{m}{2}+t)^2\big)$.
Because the Virasoro algebra modules $L\big(1, (\frac{m}{2}+t)^2\big)$ are $C_1$-cofinite, so is $M\big(1, \frac{m}{\sqrt{2}}\big)$.
\end{proof}

Consequently, the modules $M(1)^{\pm}$ are $C_1$-cofinite:
\begin{corollary}
    The irreducible modules $M(1)^\pm$ are $C_1$-cofinite.
\end{corollary}
\begin{proof}
By Theorem \ref{cofinite1}, the Heisenberg algebra $M(1)$ regarded as an $M(1)^+$-module is $C_1$-cofinite. As quotient modules of $M(1)$, $M(1)^\pm$ are also $C_1$-cofinite.
\end{proof}

Now we proceed to show the modules $M(1, \lambda)$ are $C_1$-cofinite for all $\lambda \in \mathbb{C}^\times$.
\begin{theorem}\label{cofinite2}
All the irreducible $M(1)^+$-modules $M(1, \lambda)$ for $\lambda \in \mathbb{C}^\times$ are $C_1$-cofinite.
\end{theorem}
\begin{proof}
The proof is similar to that of \cite[Theorem~16]{CMR}. For $n \in \mathbb{Z}_{\geq 0}$, let $M(1, \lambda)(n)$ be the degree $n$ subspace of  $M(1, \lambda)$.
 It is well-known that the following vectors form a basis of  $M(1, \lambda)(n)$:
 $$\alpha(-n_1)\cdots\alpha(-n_s)v_\lambda,$$
 where  $v_\lambda$ denotes the highest weight vector of $M(1, \lambda)$ and $n_1,\cdots, n_s\in \Z_{\geq 1}$ such that $n_1+\cdots+ n_s=n$.

 Set $C_1(M(1, \lambda))=\langle v_{(-1)}m|v\in M(1)^+, \wt~ v>0, m\in M(1, \lambda) \rangle,$ and $$C_1(M(1,\lambda))(n)= C_1(M(1, \lambda)) \cap M(1, \lambda)(n).$$ By the definition of $M(1, \lambda)$, every vector in $C_1(M(1,\lambda))(n)$ is a linear combination of vectors
 $$\alpha(-n_1)\cdots\alpha(-n_s)v_\lambda,\ \ n_1,\cdots, n_s\in \Z_{\geq 1}\ \text{such that } n_1+\cdots+ n_s=n$$ with coefficients in $\C[\lambda]$. Therefore, the dimension of $C_1(M(1,\lambda))(n)$ depends on the rank of a certain matrix $(a_{ij}(\lambda))$ with entries $a_{ij}(\lambda) \in \mathbb{C}[\lambda]$.

Set $$s(n)=\max \{\dim C_1(M(1,\lambda))(n)|\lambda\in \C\}.$$
Then $s(n)\leq \dim M(1,\lambda)(n)=p(n)$. By Theorem \ref{cofinite1}, $M\big(1, \frac{m}{\sqrt{2}}\alpha\big)$ is a $C_1$-cofinite  $M(1)^+$-module for any $m\in \Z_{\geq 0}$. Therefore,  there exists $n_0 \in \mathbb{Z}_{\geq 0}$ such that $s(n)=p(n)$ for any $n \geq n_0$. Thus, for all $n \geq n_0$, the dimension of $C_1(M(1,\lambda))(n)$ is strictly less than $p(n)={\rm dim}(M(1,\lambda)(n))$ only for the zero points of certain minors of the matrix $(a_{ij}(\lambda))$, which are finite sets. Let $B$ be the union of such zero points for all $n \geq n_0$. Then $B$ is a countable set. By the construction of $B$, $$C_1(M(1,\lambda))(n) = M(1,\lambda)(n)$$ holds for $\lambda \in \mathbb{C}\setminus B$ and $n \geq n_0$. Therefore, $M(1, \lambda)$ are $C_1$-cofinite for $\lambda \in \mathbb{C}\setminus B$.

Recall from Theorem 4.5 of \cite{Abe3} that there are nonzero intertwining operators of type $$\binom{M(1, \lambda+\mu)}{M(1, \lambda)\; M(1, \mu)}.$$ It follows from Main Theorem of \cite{Miy} that $M(1, \lambda+\mu)$ is $C_1$-cofinite if $M(1,\lambda)$ and $M(1,\mu)$ are both $C_1$-cofinite. Since every element $\lambda \in \mathbb{C}^\times$ can be written as a sum of two elements in $\mathbb{C}\setminus B$, it follows that $M(1,\lambda)$ is $C_1$-cofinite for all $\lambda \in \mathbb{C}^\times$.
\end{proof}

\subsection{Non-$C_1$-cofiniteness of irreducible $M(1)^+$-modules of twisted type}
The field  $h(z) = \sum _{n \in {\Z}} h(n + \frac{1}{2}) z^{-n-\frac{3}{2} }$ defines on  $M(1)(\theta)$ the structure of twisted $M(1)$--module. More precisely, for
$v = h(-m_1) \cdots h(-m_r) {\bf 1}$, set
$$W_{\theta} (v, z)  = : \partial ^{(m_1-1)} h (z) \cdots  \partial ^{(m_r-1)} h (z):, $$
where the normal ordering $:\; :$ is standard.  Define $\Delta_z$ as in \cite{DN1}. Then the  twisted vertex operator $Y_{tw} (v,z)$ for $v \in M(1)$ is given by
$$ Y_{tw} (v,z): = W_{\theta} (\Delta_z v, z). $$
Define the following subspaces of $M(1)(\theta)^{\pm}$:
$$M(1)(\theta)^{+} _{top} = \{ h(-\tfrac{1}{2}) ^{ 2 i} {\bf 1} _{tw}  \ \vert \ i \in {\Z}_{\ge 0} \},\;\;  M(1)(\theta)^{-} _{top} = \{ h(-\tfrac{1}{2}) ^{ 2 i+1}  {\bf 1}_{tw}\ \vert \ i \in {\Z}_{\ge 0} \}. $$


\begin{theorem} 
We have $$C_{1} (M(1)(\theta)^{\pm}) \cap  M(1)(\theta)^{\pm } _{top} = \{ 0\}.$$ In particular,  the $M(1)^+$-modules  $M(1)(\theta)^{\pm}$ are not $C_1$--cofinite.
\end{theorem}
 \begin{proof}
Since $M(1) ^+$ is spanned by the monomials
 $$ h(-m_1) \cdots h(-m_r){\bf 1}, $$
 where $r \in {\Z}_{\ge 0}$ is even, and $m_1, \dots, m_r \in {\Z}_{>0}$, the $C_1$-subspace $C_1(W)$ of any $M(1)^+$-module $W$ is spanned by
  $$ (h(-m_1) \cdots h(-m_r){\bf 1})_{(-1)} w, $$
  where $w \in W$. 
  
Now we consider $W = M(1)(\theta)^{\pm}$. First notice that 
$(h(-m_1) \cdots h(-m_r){\bf 1})_{(-1)}$ is  an (infinite) sum of operators of form
  $$ (*)\;\;\; :h\left(-p_1 - \frac{1}{2}\right) \cdots h\left(-p_s - \frac{1}{2}\right):, $$
where  $p_1, \dots, p_s \in {\Z}$  (not necessarily positive) such that $s$ is even, $s \le r$,   and 
$$ p_1 + p_2 + \cdots + p_s   + s/2 = m_1 + \cdots m_r.$$
This implies that $  p_1 + p_2 + \cdots + p_s \ge  r/2 >0. $
Then applying operator operator (*) on homogeneous vectors $w$ of the form 
  $$w = h\left(-l_1 - \frac{1}{2}\right) \cdots h\left(-l_q - \frac{1}{2}\right)  {\bf 1} $$ for $l_1, \dots, l_q \in {\Z}_{\ge 0 }$, $q$ even, which we call monomials in $W$, we get
$$:h\left(-p_1 - \frac{1}{2}\right) \cdots h\left(-p_s - \frac{1}{2}\right): w = w_1 + \cdots + w_s,$$ 
such that each $w_i$ is a monomial of the $L(0)$-weight 
$$l_1 + \cdots + l_q +   p_1 + p_2 + \cdots + p_s  +\frac{q+s}{2} =   l_1 + \cdots + l_q +  m_1 + \cdots m_r + \frac{q}{2}  $$    and the length  of all monomials $w_i$ is less or equal to $r+q$. This easily implies that each monomial $w_i \notin W_{top}$. We conclude that $C_1(W)$ consists of sums of monomials which don't belong to  $W_{top}$. 
Therefore, $C_{1} (W) \cap   W _{top}  = \{ 0\}$. Consequently, $W/ C_1(W)$ is an infinite-dimensional vector space which contains  a subspace isomorphic to $W_{top}$. 
\end{proof}
 
\subsection{Classification of highest weight $M(1)^+$-modules}

In order to study the extensions of irreducible modules, Abe introduced the fields  
  $$\widetilde{H}^{2r}(z) =\sum _{m \in {\Z} }{\widetilde  H} ^{2 r }    (m)  z^{-m -2 r-2}, $$ for $r \in \Z_{\geq 0}.$ In particular, operators $\widetilde{H}^{2r}(0)$, $r \in {\Z}_{>0}$, act semisimply on the irreducible modules \cite[Prop.~4.3]{Abe}.




The actions of the operators $\widetilde{H}^{4}(0)$ and $\widetilde{H}^{6}(0)$ on the top space of the irreducible modules are shown in the following table:
\begin{table}[h]
\begin{center}
\caption{}
\begin{tabular}{|c|c|c|c|c|c|c|}
\hline
$W$  &  $M(1)^+$ & $M(1)^-$ & $M(1,\lambda)$ & $M(1,\theta)^+$ & $M(1,\theta)^-$\\
\hline
$W_{\rm top} $ & $\mathbb{C}{\bf 1}$  & $\C h(-1){\bf 1}$  & $\mathbb{C}e^{\lambda}$ & $\mathbb{C}{\bf 1}_{\rm tw}$ & $\C h(-1/2){\bf 1}_{\rm tw}$\\
\hline
$\widetilde{H}^{4}(0)$ & $0$ & $1$ & $0$ & $-1/128$ & $15/128$\\
\hline
$\widetilde{H}^{6}(0)$ & $0$ & $1$ & $0$ & $1/256$ & $9/256$\\
\hline
\end{tabular}
\end{center}
\end{table}

We first need the following lemma:
\begin{lemma}\label{h4h6actondual}
    If $w$ is an $(\widetilde{H}^4(0), \widetilde{H}^6(0))$-weight vector, then its contragredient dual vector $w'$ satisfies
   \begin{equation}
        \langle \widetilde{H}^{2r}(0)w' , w\rangle = \langle w' , \widetilde{H}^{2r}(0)w\rangle,
    \end{equation}
    for $r = 2, 3$.
\end{lemma}
\begin{proof}
For a grading-restricted module $W$ and for $v \in V$, the operator $Y_{W'}(v, x)$ on the contragredient dual $W'$ is determined by the relation
\begin{equation}\label{eqn:contra}
\langle Y_{W'}(v,x)w', w \rangle = \langle w', Y_W(e^{xL(1)}(-x^{-2})^{L(0)}v,x^{-1})w \rangle
\end{equation}
for all $w \in W$ and $w' \in W'$.

Now let $V = M_\h(1,0)^+$ and $v = H^4$. Using \cite[Eqn.~(4.10)]{Abe}, we have
\[
L(1)v = \frac{1}{3}L(-3){\bf 1},\; L(1)^2v = \frac{4}{3}L(-2){\bf 1}.
\]
Plugging into \eqref{eqn:contra}, we have
\[
\langle Y_{W'}(H^4,x)w', w\rangle = x^{-8}\langle w', Y_W(e^{xL(1)}v,x^{-1})w\rangle,
\]
which implies
\begin{align*}
&    \langle Y_{W'}(H^4,x)w', w \rangle\\
& \quad = x^{-8}\langle w', Y(H^4,x^{-1})w\rangle + \frac{1}{3}x^{-7}\langle w', Y(L(-3){\bf 1},x^{-1})w\rangle + \frac{2}{3}x^{-6}\langle w', Y(L(-2){\bf 1},x^{-1})w\rangle.
\end{align*}
Recall that $\widetilde{H}^4(0) = H_3^4$. Taking coefficients of $x^{-4}$, we have
\[
\langle \widetilde{H}^4(0)w', w\rangle = \langle w', \widetilde{H}^4(0)w\rangle + \frac{1}{3}\langle w', (L(-3){\bf 1})_2w\rangle+\frac{2}{3}\langle w', L(0)w\rangle,
\]
which implies
\[
\langle \widetilde{H}^4(0)w', w\rangle = \langle w', (\widetilde{H}^4(0)-2/3L(0)+2/3L(0))w\rangle = \langle w', \widetilde{H}^4(0)w\rangle.
\]
The same spirit of calculations also shows
\[
\langle \widetilde{H}^6(0)w', w \rangle = \langle w', \widetilde{H}^6(0)w \rangle,
\]
we omit the details here.
\end{proof}

Now we are ready to prove the main theorem of this section, the proof relies on some results from \cite[Sec.~4]{Abe}.
\begin{theorem}\label{thm:irreduciblehwm}
    Suppose $M$ is a highest weight module with the highest weight vector $v$, then
    \begin{itemize}
        \item[(1)]If the conformal weight of $v$ is nonzero, then $M$ is irreducible.
        \item[(2)]If the conformal weight of $v$ is zero, then either $M$ is isomorphic to $M(1)^+$ or it has length $2$ with the following non-split exact sequence
        \[
        0 \rightarrow M(1)^- \rightarrow M \rightarrow M(1)^+ \rightarrow 0.
        \]
    \end{itemize}
\end{theorem}
\begin{proof}
Let $M$ be a highest weight module and let $L$ be its simple quotient. Then
we have the following exact sequence
\begin{equation}\label{eqn:basicexact}
0 \rightarrow K \rightarrow M \rightarrow L \rightarrow 0.
\end{equation}
Applying the contragredient dual, we obtain
\[
0 \rightarrow L' \rightarrow M' \rightarrow K' \rightarrow 0,
\]
Note that the simple modules are self-dual so that $L' \cong L$.

Suppose $K$ is nonzero, since it is grading-restricted, its top space is finite dimensional and thus has a $1$-dimensional irreducible $A(M(1)^+)$-submodule with a generator $u \in K$. Its dual vector $u'$ in $M'$ will be a subsingular vector, in particular, $L(n)u' \in L'\cong L$ for all $n \in \Z_{\geq 1}$. Note also that $u'$ can not be singular, otherwise, take an element $w \in K \subset M$, because $M$ is a highest weight module of highest weight $v$, $w$ is a linear combination of vectors
\[
w = J_{-i_1}\cdots J_{-i_k}L(-j_1)\cdots L(-j_t)v,
\]
where $i_1 \geq \dots \geq i_k \geq -2$, $j_1 \geq \dots \geq j_t \geq 1$ for some $k, t \in \Z_{\geq 0}$ not both $0$. If $u'$ is singular, then $\langle u', w\rangle$ is a linear combination of
\begin{align*}
\langle u', J_{-i_1}\cdots J_{-i_k}L(-j_1)\cdots L(-j_t)v \rangle= \langle L(j_t)\cdots L(j_1)J_{i_k+6}\cdots J_{i_1+6}u', v \rangle = 0.
\end{align*}
Since $w$ is an arbitrary element in $K$, it implies $u' = 0$ which is a contradiction.

Write $L(1)u' = \sum_i u_i$, where $u_i \in L$ are simultaneous eigenvectors for $\widetilde{H}^4(0)$ and $\widetilde{H}^6(0)$ with eigenvalues pairs $(h_i, k_i)$. By \cite[Prop.~4.3]{Abe} (cf. \cite[Eqn.~(4.14)]{Abe}),
\begin{equation}\label{eqn:Abe414}
    (h_i,k_i) \in \Z_{\geq 0}\times \Z_{\geq 0}, \quad {\rm or} \quad \bigg(-\frac{1}{128}+ \frac{1}{8}\Z_{\geq 0}\bigg)\times \bigg(\frac{1}{256}+\frac{1}{32}\Z_{\geq 0}\bigg).
\end{equation}

Since as an $A(M(1)^+)$-module, $\C u$ is isomorphic to one listed in Table 1, correspondingly the $(\widetilde{H}^4(0), \widetilde{H}^6(0))$-weight of $u$ must be one of the pairs in Table 2, which we denote by $(p,q)$. By Lemma \ref{h4h6actondual}, the $(\widetilde{H}^4(0), \widetilde{H}^6(0))$-weight of $u'$ is the same as $(p,q)$. By \cite[Prop.~4.5]{Abe}, the pair $(h_i,k_i)$ of $(\widetilde{H}^4(0), \widetilde{H}^6(0))$-weight of $u_i$ satisfies the equation
\[
(5(h_i-p)(h_i-p-1) +9(k_i-q)-1)L_1u' = 0.
\]

As the proof of \cite[Lemma~4.8]{Abe}, we will prove
\begin{equation}\label{eqn:nonzero}
5(h - p)(h - p - 1) + 9(k - q) - 1 \neq 0
\end{equation}
for all $(h,k) \in \C^2$ satisfying \eqref{eqn:Abe414}. 

If the conformal weight of the highest weight vector of $M$ is nonzero, the conformal weight of $u$ must be strictly greater than $1$. It follows that $\C u$ is isomorphic to the top of $M(1,\lambda)$ for some $\lambda$ as an $A(M(1)^+)$-module (see Table 1). In particular, $(p,q) = (0,0)$. Then the inequality \eqref{eqn:nonzero} becomes
\begin{equation}\label{eqn:nonzeoro1}
    5h(h-1)+9k-1 \neq 0,
\end{equation}
which always holds for $(h,k) \in \Z_{\geq 0}\times \Z_{\geq 0}$. Now we prove it also holds for
\[
h = -\frac{1}{128} + \frac{a}{8},\;\;\; k = \frac{1}{256} + \frac{b}{32}
\]
for $(a,b) \in \Z_{\geq 0}\times \Z_{\geq 0}$. Otherwise by plugging into the left-hand-side of \eqref{eqn:nonzeoro1}, we have
\[
5\cdot 2^8a^2 - 5(2^{11}+2^5)a +9\cdot 2^9b + 5\cdot 2^7 + 9\cdot 2^6 -2^{14} + 5 = 0,
\]
note the left-hand-side is always odd if $a,b$ are both integers, so the equation cannot have integer solutions. It follows that \eqref{eqn:nonzero} always holds for $(h,k) \in \C^2$ satisfying \eqref{eqn:Abe414}, which means $L(1)u' =0$, and furthermore Abe's proof of \cite[Lemma~4.8]{Abe} also shows $L(2)u' = 0$. Therefore, $u'$ is a singular vector, which is a contradiction! It follows that $K = 0$, and $M$ is irreducible.

Now we proceed to prove $(2)$. If $L(-1)v=0$, then $v$ is a vacuum-like vector. By Proposition 4.7.7 of \cite{LL}, $M$ is isomorphic to $M(1)^+$. 
We next consider the case that $L(-1)v\neq 0$. First, the weight one subspace $M_1$ of $M$ is spanned by $L(-1)v$ and $J_2v$. Let $K$  be the maximal submodule of $M$. Then we have
    \begin{equation}\label{eqn:basicexact1}
    0 \rightarrow K \rightarrow M \rightarrow M(1)^+ \rightarrow 0.
    \end{equation}
    This implies that $L(-1)v\in K$ and $J_2v\in K$. Furthermore, by the discussion above, $K$ is generated by $L(-1)v$.

    If $\dim M_1=1$, then $K$ must be isomorphic to $M(1)^-$. Therefore, we have \begin{equation}
    0 \rightarrow M(1)^- \rightarrow M \rightarrow M(1)^+ \rightarrow 0.
    \end{equation}
If $\dim M_1=2$, then $M_1$ has a 1-dimensional $A(M(1)^+)$-submodule $\mathbb C w$. Furthermore, the $M(1)^+$-submodule $K^1$ of $K$ generated by $w$ must be isomorphic to $M(1)^-$. Moreover, $K/K^1$ must be isomorphic to $M(1)^-$. This implies that $K\cong M(1)^-\oplus M(1)^-$. However, this implies that $L(-1)v$ is a highest weight vector, and $K$ is isomorphic to $M(1)^-$, which is a contradiction. We proved $(2)$.
\end{proof}

\begin{remark}
    The length $2$ highest weight module in Theorem \ref{thm:irreduciblehwm} was constructed in \cite[Sec.~6]{Abe}.
\end{remark}

\section{Tensor category structure on the category of $C_1$-cofinite $M(1)^+$-modules}
\subsection{Tensor category of $C_1$-cofinite $M(1)^+$-modules}
Denote by $\mathcal{C}_1(M(1)^+)$ the category of $C_1$-cofinite grading-restricted generalized $M(1)^+$-modules. We have shown the irreducible modules $M(1)^\pm$ and $M(1, \lambda)$ for $\lambda \in \mathbb{C}^\times$ are all $C_1$-cofinite, and thus by \cite[Prop.~2.9]{H3}, all the finite length modules with simple composition factors $M(1)^\pm$ and $M(1, \lambda)$ are in the category $\mathcal{C}_1(M(1)^+)$. In the next theorem, we show the category $\mathcal{C}_1(M(1)^+)$ is exactly the category of finite length modules with simple composition factors $M(1)^\pm$ and $M(1, \lambda)$. As a consequence, the category $\mathcal{C}_1(M(1)^+)$ possesses a structure of braided and vertex tensor category.

We first prove the following lemma:
\begin{lemma}\label{projtwisttype}
    The modules $M(1)(\theta)^{\pm}$ are projective in the category of finite length $M(1)^+$-modules. 
\end{lemma}
\begin{proof}
    Let $P$ be either $M(1)(\theta)^+$ or $M(1)(\theta)^-$. We show by induction on the length of the $M(1)^+$-module $N$ that the following short exact sequence splits:
    \[
    0 \rightarrow M \rightarrow N \rightarrow P \rightarrow 0.
    \]
    If $N$ has length 2, the claim follows from \cite[Theorem~5.5, 5.6]{Abe} (cf. Theorem \ref{Abe}). In general, we take an irreducible submodule $M_0 \subset M$ and consider the exact sequence
    \[
    0 \rightarrow M/M_0 \rightarrow N/M_0 \rightarrow P \rightarrow 0,
    \]
    which splits by induction. Let $\widetilde{P}$ be a submodule of $N$ such that $\widetilde{P}/M_0 = P$. By \cite[Theorem~5.5, 5.6]{Abe} again, $\widetilde{P} = M_0 \oplus P$. Thus $P$ is a submodule of $N$, and the exact sequence splits.
\end{proof}

\begin{theorem}\label{thm:mainM1+}
The category $\mathcal{C}_1(M(1)^+)$ is the same as the category of finite length modules with simple composition factors isomorphic to $M(1)^\pm$ and $M(1, \lambda)$ for $\lambda \in \C^\times$.
\end{theorem}
\begin{proof}
It suffices to show a $C_1$-cofinite grading-restricted generalized module $M$ is finite length whose simple composition factors are  $M(1)^\pm$ and $M(1, \lambda)$ for $\lambda \in \mathbb{C}^\times$.

We consider the top space $M(0)$ of $M$, since $M$ is lower bounded, $M(0) \neq 0$. From Zhu's theorem (\cite{Z}), $M(0)$ is an $A(M(1)^+)$-module. Because $A(M(1)^+)$ is commutative, we can take a one-dimensional irreducible submodule of $M(0)$, and let $M_0$ be the highest weight module generated by this irreducible $A(M(1)^+)$-submodule. If $M/M_0$ is zero, we are done since all the highest weight modules are of finite length, otherwise, consider the quotient module $M/M_0$, and take an irreducible $A(M(1)^+)$-submodule in its top space. It generates an $M(1)^+$-submodule of $M/M_0$, say $M_1/M_0$. Performing this procedure again on the quotient $M/M_1$ if its not zero, we obtain a submodule $M_2/M_1$. Repetition of this procedure yields a filtration
\[
 \cdots \supset M_n \supset \cdots \supset M_2 \supset M_1 \supset M_0 \supset 0,
\]
where $M_{i+1}/M_i$ are all highest weight modules. This filtration series will terminate because the dimensions of the $C_1$-quotient spaces strictly decrease:
\[
{\rm dim}(M/M_{i+1}/C_1(M/M_{i+1})) < {\rm dim}(M/M_{i}/C_1(M/M_{i})).
\]
Now we conclude that $M$ is of finite length because the length of filtration series is finite and all the filtration factors are of finite length by Theorem \ref{thm:irreduciblehwm}.

If for some $i$, $M_{i+1}/M_{i} \cong M(1)(\theta)^+$ or $M(1)(\theta)^{-}$, because $M(1)(\theta)^\pm$ are self-dual, $(M_{i+1}/M_i)' \cong M_{i+1}/M_i$. Consider the short exact sequence
\[
0 \rightarrow M_{i+1}/M_i \rightarrow M/M_i \rightarrow M/M_{i+1} \rightarrow 0
\]
and its contragredient dual, by Lemma \ref{projtwisttype}, $M_{i+1}/M_i$ is a direct summand of $(M/M_i)'$, and thus $M_{i+1}/M_i$ is also a direct summand of $M/M_i$. It is a contradiction since $M/M_i$ is $C_1$-cofinite but $M_{i+1}/M_i$ is not $C_1$-cofinite. As a result, the composition factors should all be of untwisted type.
\end{proof}

Combining Theorem \ref{thm:mainM1+} and \cite[Theorem~3.6]{CY} (cf. \cite[Theorem~2.25]{CMOY}), we establish the vertex and braided tensor category structure on the category $\mathcal{C}_1(M(1)^+)$:
\begin{theorem}
The category $\mathcal{C}_1(M(1)^+)$ is a vertex and braided tensor category.
\end{theorem}

\subsection{Rigidity and fusion product decompositions}
In this section, we will determine the fusion products of simple objects in the category $\mathcal{C}_1(M(1)^+)$ of $C_1$-cofinite grading-restricted generalized $M(1)^+$-modules. Using the fusion product decompositions, \cite[Theorem~1.1]{EP} and \cite[Theorem~4.4.1]{CMY2}, we prove the category $\mathcal{C}_1(M(1)^+)$ is rigid. 

First, we calculate fusion rules involving $M(1)^-$. Since $M(1)^+$ is the $\mathbb{Z}_2$-orbifold of $M(1)$ and $M(1) = M(1)^+ \oplus M(1)^-$, we have the following result by \cite{DLM, M1}.
\begin{proposition}\label{fusionM1}
$M(1)^- \boxtimes M(1)^- = M(1)^+.$
\end{proposition}

Next we calculate the fusion rules between $M(1)^-$ and $M(1, \lambda)$. For simplicity, from now on we denote the category $\mathcal{C}_1(M(1)^+)$ by $\mathcal{C}$. Note the Heisenberg algebra $M(1)$ can be regarded as a commutative associative algebra in the category $\cC$, denote this algebra by $A$. Let $\cC_A$ be the category of $A$-modules in $\cC$, and let $\cF_A$ be the induction functor $\cC \rightarrow \cC_A$, and $\cG$ be the restriction functor by $\cC_A \rightarrow \cC$.
For $\lambda \in \C^{\times}$, the module $M(1, \lambda)$ can be regarded as both $M(1)$-module and $M(1)^+$-module. To avoid confusion, we denote the $A$-module $M(1, \lambda)$ by $\cF_\lambda$, and still use $M(1, \lambda)$ for $M(1)^+$-modules. Then we have the following result:
\begin{proposition}\label{fusionM2}
For $\lambda \in \C^{\times}$, $M(1)^- \boxtimes M(1, \lambda)= M(1, \lambda).$
\end{proposition}
\begin{proof}

By Proposition \ref{fusionM1}, we have
\[
M(1)\boxtimes M(1) = M(1) \oplus M(1) \boxtimes M(1)^- = M(1) \oplus M(1)
\]
as an object in $\cC$. By Frobenius reciprocity,
    \[
    \Hom_{\cC_A}(\cF_A(A), A) = \Hom_{\cC}(\cG(A), \cG(A)) = \mathbb{C}^2.
    \]
 This forces that $\cF_A(A) = A \oplus A$ as $A$-modules.
For $\lambda \in \C^\times$, it follows from \cite[Lemma~2.8]{CLR} that
\[
A \boxtimes M(1,\lambda) = \cG(\cF_A(A) \boxtimes_A \cF_{\lambda}) =  \cG(\cF_{\lambda} \oplus \cF_\lambda) = M(1,\lambda) \oplus M(1,\lambda).
\]
Moreover, by Frobenius reciprocity,
    \[
    \Hom_{\cC_A}(\cF_A(M(1,\lambda)), \cF_{\pm \lambda}) = \Hom_{\cC}(M(1,\lambda), \cG(\cF_{\pm \lambda}))  = \mathbb{C}.
    \]
   This forces  that
    \begin{equation}\label{eqn:inducedmod}
        \cF_A(M(1,\lambda)) = \cF_{\lambda} \oplus \cF_{-\lambda}.
    \end{equation}
    In particular, we have $M(1)^- \boxtimes M(1,\lambda)= M(1,\lambda).$
\end{proof}

We are now ready to prove the main result of this section.
\begin{theorem} \label{fr-1}
For $\lambda, \mu \in \mathbb{C}^{\times}$,
\begin{align*}
&M(1)^- \boxtimes M(1)^- = M(1)^+,\\
&M(1)^- \boxtimes M(1,\lambda)= M(1,\lambda),\\
& M(1,\lambda) \boxtimes M(1, \mu) = M(1, \lambda+\mu) \oplus M(1, \lambda - \mu).
\end{align*}
\end{theorem}
\begin{proof}
It remains to prove the third identity. By (\ref{eqn:inducedmod}), we have
$$\cF_A(M(1,\lambda)) = \cF_{\lambda} \oplus \cF_{-\lambda}.$$
It follows from \cite[Lemma~2.8]{CLR} that
\begin{align*}
&M(1,\lambda) \boxtimes M(1, \mu) = M(1,\lambda) \boxtimes \cG(\cF_{\mu})\\ &\cong \cG(\cF_A(M(1,\lambda))\boxtimes_A \cF_\mu) = \cG((\cF_\lambda\oplus \cF_{-\lambda})\boxtimes \cF_\mu)\\
&  = \cG(\cF_{\lambda + \mu} \oplus \cF_{\lambda - \mu}) = M(1, \lambda + \mu) \oplus M(1, \lambda-\mu),
\end{align*}
as desired.
\end{proof}

\begin{remark}
   Theorem \ref{fr-1} gives an alternative proof of the fusion rule results from Abe's results \cite{Abe}, which only indicate that there is a surjective map from $M(1,\lambda)\boxtimes M(1,\mu)$ to $M(1,\lambda+\mu) \oplus M(1,\lambda-\mu)$. In a non-semisimple category, Abe's results are still not quite the same as Theorem \ref{fr-1}.
\end{remark}

Using the fusion rules given in Theorem \ref{fr-1}, we can verify the conditions of \cite[Theorem~1.1]{EP}, which shows that all the simple objects in the category $\cC$ are rigid. Then it follows from \cite[Theorem~4.4.1]{CMY2} that $\cC$ is rigid:
\begin{theorem}
    The category $\cC$ of grading-restricted $C_1$-cofinite $M(1)^+$-modules is rigid.
\end{theorem}
\begin{proof}
By \cite[Theorem~1.1]{EP}, which says all the non-negligible objects of moderate growth in a braided tensor category are rigid, we only need to show $M(1, \lambda)$ for all $\lambda \in \C^\times$ are non-negligible of moderate growth. 

Note that an indecomposable object $X$ is called non-negligible if there exists an object $Y$ such that ${\bf 1}$ is a direct summand of $Y\otimes X$. An object $X$ is of moderate growth if there exists $n \in \Z_{\geq 1}$ such that dim End$(X^{\otimes n}) < n!$.

From the fusion product decomposition $$M(1,\lambda) \boxtimes M(1, \lambda) = M(1, 2\lambda) \oplus M(1) = M(1, 2\lambda) \oplus M(1)^+ \oplus M(1)^-,$$ the unit object ${\bf 1}$ is a direct summand of $M(1,\lambda) \boxtimes M(1, \lambda)$, which means $M(1, \lambda)$ is non-negligible. Furthermore, since
\[
M(1,\lambda)^{\boxtimes 6} = M(1,6\lambda) \oplus 6M(1,4\lambda) \oplus 15M(1,2\lambda)\oplus 10M(1)^+ \oplus 10M(1)^-,
\]
dim End$(M(1,\lambda)^{\boxtimes 6}) = 462 < 6! = 720$, it follows that $M(1,\lambda)$ has moderate growth.

Now all the simple objects in the category $\cC$ are rigidity, then by \cite[Theorem~4.4.1]{CMY2} which says that if all the simple objects in the category of finite length modules are rigid, then the category is rigid, we know the category $\cC$ is rigid.
\end{proof}

\subsection{Connection to a non-semisimple Heisenberg tensor category}
Now we can determine the category $\cC_A$ of $A$-modules in the category $\cC$:
\begin{proposition}
    The category $\cC_A$ is the category of finite length $M(1)$-modules whose composition factors are the Fock modules $\cF_{\lambda}$ for $\lambda \in \C$.
\end{proposition}
\begin{proof}
   Because the category $\cC$ is a braided and vertex tensor category, the assumptions in \cite[Theorem~4.15]{M2} are all satisfied, it follows that every indecomposable object of $\cC_A$ is a $\theta^i$-twisted $M(1)$-module for $i = 0,1$. We will show an object $M$ in $\cC_A$ cannot be a $\theta$-twisted module, otherwise, we have that
   \begin{equation}\label{g-twist}
   \mu_M (\theta \boxtimes {\rm id}_M) \mathcal{M}_{A, M} = \mu_M, 
   \end{equation}
   where $\mu_{-}: A \boxtimes - \rightarrow -$ is the multiplication morphism in the category $\cC$, $\mathcal{M}_{A, -} : = c_{-, A}\circ c_{A, -}$ is the monodromy isomorphism. By the fusion product decomposition in Theorem \ref{fr-1} and the balancing axiom, we can deduce that $\mathcal{M}_{M(1)^-, M_0} = \id_{M(1)^-\boxtimes M_0}$ for a minimal irreducible submodule $M_0$ of $M$. It follows that $\mu_{M}|_{M(1)^- \boxtimes M_0} = - \mu_{M}|_{M(1)^- \boxtimes M_0}$, which is a contradiction.
   
   It is also easy to see all the simple $M(1)$-modules $\mathcal{F}_\lambda$ for $\lambda \in \C$ lie in $\cC_A$, thus $\cC_A$ consists of finite length modules with simple composition factors $\mathcal{F}_\lambda$ for $\lambda \in \C$.
\end{proof}

\begin{remark}
The existence of the tensor category structure of finite length $M(1)$-modules with simple composition factors $\cF_{\lambda}$ for $\lambda \in \C$ is evident because the category of finite length $M(1)$-modules is exactly the category of $C_1$-cofnite modules which follows from the fact that every highest weight $M(1)$-module is isomorphic to $\mathcal{F}_{\lambda}$ for some $\lambda \in \C$. This category should be the category {\rm Rep} $\C[x]$ of finite dimensional modules for $\C[x]$, where $x$ is the zero module operator $\alpha(0)$.
\end{remark}

Let $(\cC_A)^{\Z_2}$ be the $\Z_2$-equivariantization of the category $\cC_A$, the induction functor actually maps into the category $(\cC_A)^{\Z_2}$ (see \cite[Sec.~2.3]{M2}). Again, because the category $\cC$ is a braided and vertex tensor category, the assumptions in \cite[Theorem~4.17]{M2} are all satisfied, it follows that the induction functor $\cF_A: \cC \rightarrow (\cC_A)^{\Z_2}$ is an equivalence of braided tensor categories:  
\begin{corollary}\label{cor:equiv}
    The categories $\cC$ is braided equivalent to the category $(\cC_A)^{\Z_2}$.
\end{corollary}

\begin{remark}
We should also mention that in  \cite[Section 6] {GLM}, the authors  consider an infinite
Tambara-Yamagami category ${\rm Vect}_{{\Bbb R}^d} \oplus {\rm Vect}$, which is semisimple  and include all untwisted and twisted modules of the Heisenberg VOA.  Our category $\cC$ is not semisimple and only has untwisted type modules, so it is different from the category constructed in \cite{GLM}.
\end{remark}

\section{Minimal affine $W$-algebra $W_{-1}^{min}(sp(2n))$}
\def\theequation{3.\arabic{equation}}
\setcounter{equation}{0}
In this section, we will prove the minimal $W$-algebra $W_{-1}^{min}(sp(2n))$ is an extension of $L_{-\frac{1}{2}}(sp(2n))\otimes M(1)^+$, and then determine the decompositions of irreducible $W_{-1}^{min}(sp(2n))$-modules as $L_{-\frac{1}{2}}(sp(2n-2)) \otimes M(1) ^+$-modules. These structures will play an important role in the study of the category of grading-restricted generalized modules for $L_{-1}(sp(2n))$.

\subsection{$W_{-1}^{min}(sp(2n))$ as an extension of $L_{-\frac{1}{2}}(sp(2n))\otimes M(1)^+$}\label{weyl}

We first introduce the Weyl vertex algebra $\mathcal{S}(n)$, which is also called a rank $n$ $\beta\gamma$ system in the literature (see \cite{K2}, \cite{W}). Recall the Weyl algebra $W_{n}$ is an associative algebra with generators
 $$a_i^{\pm}(r), \quad r\in\frac{1}{2}+\Z,\; 1\leq i\leq n,$$
satisfying the nontrivial relations
 $$[a_i^+(r), a^-_j(s)]=\delta_{r+s, 0}\delta_{i,j}.$$
 Denote by $\mathcal{S}(n)$ the irreducible module of $W_{n}$ generated by $\1$ such that $$a_i^{\pm}(r)\1=0, \; r> 0.$$
 The module $\mathcal{S}(n)$ has a simple vertex algebra structure, called the Weyl vertex algebra or a rank $n$ $\beta\gamma$-system, with the Virasoro element $$\omega_1=\frac{1}{2}\sum_{i=1}^{n}\left(a_{i}^-\left(-\frac{3}{2}\right)a_{i}^+\left(-\frac{1}{2}\right)-a_{i}^+\left(-\frac{3}{2}\right)a_{i}^-\left(-\frac{1}{2}\right)\right)\1,$$ and the vertex operator $Y(\cdot, z)$ is determined by for $1\leq i\leq n$,
 $$Y\left(a_{i}^{\pm}\left(-\frac{1}{2}\right)\1, z\right)=\sum_{m\in \Z}a_i^{\pm}\left(m+\frac{1}{2}\right)z^{-m-1}.$$
Set $Y(\omega_1, z)=\sum_{n\in Z}L(n)z^{-n-2}$. Then $\mathcal{S}(n)$ is $\frac{1}{2}\Z$-graded with respect to the operator $L(0)$. For $i\in \{0,1\}$, set
 $$\mathcal{S}(n)^{\bar i}=\big\{v\in \mathcal{S}(n)| L(0)v=\big(\frac{i}{2}+k\big)v \text{ for some } k\in \Z\big\}.$$
The following result has been proved in \cite{FF} (cf. \cite[Theorem~2.1]{W}):
\begin{proposition}
We have
\begin{itemize}
\item[(1)] As vertex operator algebras, $\mathcal{S}(n)^{\bar 0} \cong L_{-\frac{1}{2}}(sp(2n))$.
\item[(2)] As $L_{-\frac{1}{2}}(sp(2n))$-modules, $\mathcal{S}(n)^{\bar 1}\cong L_{sp(2n)}\big(-\frac{1}{2},\bar{\Lambda}_1\big)$, where $\bar{\Lambda}_1$ is the first fundamental weight of $sp(2n)$.
\end{itemize}
\end{proposition}




We construct a subalgebra of $\mathcal{S}(n-1)\otimes M(1)$ in the following Lemma:
\begin{lemma}\label{simple}
 Let $\lambda_1$ be the first fundamental weight of $sp(2n-2)$. Then the subspace
   \[
   U: = L_{sp(2n-2)}\left(-\frac{1}{2}, 0\right) \otimes M(1) ^+ \oplus  L_{sp(2n-2)}\left(-\frac{1}{2}, \lambda_1\right)  \otimes M(1)^- 
   \]
   of $\mathcal{S}(n-1)\otimes M(1)$ is a simple vertex operator algebra.
   \end{lemma}
   \begin{proof}
   The subspace $U$ is actually the fixed point subalgebra of $\mathcal{S}(n-1)\otimes M(1)$ by the autumorphism $\sigma \otimes \theta$, where $\sigma$ be the linear isomorphism of $\mathcal{S}(n-1)$ defined by $\sigma|_{\mathcal{S}(n-1)^{\bar 0}}=\id$ and $\sigma|_{\mathcal{S}(n-1)^{\bar 1}}=-\id$. 
   
   Since $\mathcal{S}(n-1)\otimes M(1)$ is a simple vertex algebra, it follows from \cite[Theorem~4.4]{DM} that $U$ is a simple vertex operator algebra.
   \end{proof}

\begin{theorem}\label{structureA}
The minimal affine $W$-algebra  $$W_{-1}^{min}(sp(2n))\cong L_{sp(2n-2)}\left(-\frac{1}{2},0\right) \otimes M(1) ^+ \oplus  L_{sp(2n-2)}\left(-\frac{1}{2},\lambda_1\right)  \otimes M(1)^-$$ as a vertex operator algebra.
  \end{theorem}
 \begin{proof}
 It is enough to verify all the conditions in \cite[Theorem~3.2]{ACKL} hold. 
 
 First, we show the central charges of $W_{-1}^{min}(sp(2n))$  and $U$ are equal. It is known that the central charges of $L_{sp(2n-2)}\big(-\frac{1}{2},0\big)$ and $M(1)$ are $-n+1$ and $1$, respectively. This implies that the central charge of $U$ is equal to $-n+2$. On the other hand, the central charge of $W_{-1}^{min}(sp(2n))$ is also equal to $-n+2$ (cf. \cite[Eqn.~(5.7)]{KW2}). Hence, the central charges of $W_{-1}^{min}(sp(2n))$ and $U$ are equal.

By \cite[Theorem~5.1]{KW2}, the affine subalgebra of $W_{-1}^{min}(sp(2n))$ is a quotient of $L_{-\frac{1}{2}}(sp(2n-2))$ (cf. \cite[Table~3]{ArM}). Therefore, the second condition in \cite[Theorem~3.2]{ACKL} holds.

We now show that $W_{-1}^{min}(sp(2n))$ and $U$ have the same generators. In fact, \cite[Theorem~5.1]{KW2} says $W_{-1}^{min}(sp(2n))$ is strongly generated by the set
\begin{align*}
    x(-1)\1\otimes \1, \;\; \omega_1\otimes \1+\1\otimes \omega_2,\;\; y\otimes h(-1)\1,
\end{align*}
for $x\in sp(2n-2)$, $y\in L_{sp(2n-2)}(\lambda_1)$, and $\omega_1$ and $\omega_2$ are the conformal elements of $L_{-\frac{1}{2}}(sp(2n-2))$ and $M(1)^+$, respectively.
In the following Proposition \ref{prop:genofU} with proof deferred to Appendix \ref{appB}, we prove $U$ is also generated by the same set of generators. 

Finally, by Lemma \ref{simple}, $U$ is a simple vertex algebra. Thus all the conditions in \cite[Theorem~3.2]{ACKL} hold. Consequently $W_{-1}^{min}(sp(2n))$ is isomorphic to $U$ as a vertex operator algebra.
  \end{proof}

\begin{proposition}\label{prop:genofU}
  The vertex operator algebra $U$ is generated by 
  \begin{align*}
    &x(-1)\1\otimes \1,~~~x\in sp(2n-2),\\
    &\omega_1\otimes \1+\1\otimes \omega_2, \\
    &y\otimes h(-1)\1,~~~y\in L_{sp(2n-2)}(\lambda_1),
    \end{align*}
    where $\omega_1$ and $\omega_2$ are the conformal elements of $L_{-\frac{1}{2}}(sp(2n-2))$ and $M(1)^+$, respectively.
  \end{proposition}

 \subsection{Decompositions of $W_{-1}^{min}(sp(2n))$-modules as $L_{-\frac{1}{2}}(sp(2n-2)) \otimes M(1)^+$-modules}  In this section, we will decompose certain irreducible $W_{-1}^{min}(sp(2n))$-modules as $L_{-\frac{1}{2}}(sp(2n-2)) \otimes M(1) ^+$-modules.  For this, we need to use the theory of Zhu's functor developed in \cite{Z}.

We first determine the Zhu's algebra  $A(W_{-1}^{min}(sp(2n)))$ for the minimal $W$-algebra $W_{-1}^{min}(sp(2n))$.
\begin{proposition}\label{ZhuW}
  The Zhu's algebra $A(W_{-1}^{min}(sp(2n)))$ is isomorphic to $$ (U(sp(2n-2))  \otimes {\C}[\bar\omega_2] )/I,$$  where $I$ is a certain two sided ideal of  $U(sp(2n-2))  \otimes {\C}[\bar\omega_2]$.
\end{proposition}
\begin{proof}
Let $L(1,0)$ be the Virasoro vertex operator algebra of central charge $1$. By Theorem \ref{structureA}, $L_{-\frac{1}{2}}(sp(2n-2)) \otimes L(1,0)$ is a subalgebra of $W_{-1}^{min}(sp(2n))$. By the definition of the Zhu's algebra, there exists an algebra homomorphism $$\phi: A \big(L_{-\frac{1}{2}}(sp(2n-2)) \otimes L(1,0)\big)\to A(W_{-1}^{min}(sp(2n))).$$

By \cite[Theorem~5.1]{KW2}, $W_{-1}^{min}(sp(2n))$ is generated by
\begin{align*}
    &x(-1)\1\otimes \1,~~~x\in sp(2n-2),\\
    &\1\otimes \omega_2,\\
    &y\otimes h(-1)\1,~~~y\in L_{sp(2n-2)}(\bar{\lambda}_1).
    \end{align*}
It follows from Proposition 2.5 of \cite{Abe2} that $A(W_{-1}^{min}(sp(2n)))$ is generated by
\begin{align*}
    &\overline{x(-1)\1\otimes \1},~~~x\in sp(2n-2),\\
    &\overline{\1\otimes \omega_2},\\
    &\overline{y\otimes h(-1)\1},~~~y\in L_{sp(2n-2)}(\bar{\lambda}_1).
    \end{align*}
Moreover, since $\wt~ y\otimes h(-1)\1=\frac{3}{2}$, we have
\begin{align*}
    (y\otimes h)\circ \1=\Res_z\frac{(1+z)}{z^2}Y(y\otimes h, z)\1=y\otimes h.
    \end{align*}
Therefore, $\overline{y\otimes h(-1)\1}=0$, for any $y\in L_{sp(2n-2)}(\bar{\lambda}_1)$. This implies that $\phi$ is surjective. By \cite[Theorem~3.1.1]{FZ} and \cite[Lemma~4.1]{Wa}, $$A(W_{-1}^{min}(sp(2n))) \cong (U(sp(2n-2))  \otimes {\C}[\bar\omega_2] )/I,$$  where $I$ is a certain two sided ideal of  $U(sp(2n-2))  \otimes {\C}[\bar\omega_2]$.
\end{proof}

By Proposition \ref{ZhuW} and using the theorem of Zhu's functor (\cite{Z}), we have:
\begin{corollary}\label{isoM}
Let $W_1, W_2$ be irreducible $W_{-1}^{min}(sp(2n))$-modules. Suppose that
\begin{itemize}
\item[(1)] $W_1, W_2$ have the same conformal weights,
\item[(2)] The top spaces $W_1(0), W_2(0)$ are isomorphic as $sp(2n-2)$--modules.
\end{itemize}
Then $W_1$ and $W_2$ are isomorphic.
\end{corollary}


Because the functor $H_{\theta}$ is exact (cf. Theorem \ref{exact}), $H_{\theta}(L_{sp(2n)}(-1,\bar{\Lambda}_2))$ and $H_{\theta}   (L_{sp(2n)}(-1,s\bar{\Lambda}_1) )$ for the first two fundamental weights $\bar{\Lambda}_1$ and $\bar{\Lambda}_{2}$ of $sp(2n)$ and $s \in \Z_{\geq 0}$, are irreducible $W_{-1}^{min}(sp(2n))$-modules. We will determine the decompositions of these irreducible $W_{-1}^{min}(sp(2n))$-modules as $L_{-\frac{1}{2}}(sp(2n-2)) \otimes M(1) ^+$-modules.


\begin{theorem}\label{structureM}
   As $L_{-\frac{1}{2}}(sp(2n-2)) \otimes M(1) ^+$-modules, we have
  $$H_{\theta} (L_{sp(2n)}(-1,\bar{\Lambda}_2)) \cong  L_{sp(2n-2)}\left(-\frac{1}{2},0\right) \otimes M(1) ^- \oplus  L_{sp(2n-2)}\left(-\frac{1}{2},\lambda_1\right) \otimes M(1) ^+.$$
  \end{theorem}
 \pf
  Since $H_{\theta}(L_{sp(2n)}(-1,\bar{\Lambda}_2))$ is an irreducible  $W_{-1}^{min}(sp(2n))$-module, by Proposition \ref{ZhuW}, $H_{\theta}   (L_{sp(2n)}(-1,\bar{\Lambda}_2))(0)$ is an irreducible $U(sp(2n-2))  \otimes {\C}[\bar\omega_2]$-module. It follows that $H_{\theta}   (L_{sp(2n)}(-1,\bar{\Lambda}_2))(0)$ is an irreducible $U(sp(2n-2))$-module. By \cite[Theorem~6.3]{KW2}, $H_{\theta}(L_{sp(2n)}(-1,\bar{\Lambda}_2))(0) \cong L_{sp(2n-2)}(\lambda_1)$ as an $sp(2n-2)$-module. Moreover, By \cite[Theorem~7.1]{KW2}, the conformal weight of $H_{\theta} (L_{sp(2n)}(-1,\bar{\Lambda}_2))$ is $\frac{1}{2}$.

  On the other hand, recall that $\sigma\otimes \theta$ is an automorphism of $\mathcal{S}(n-1)\otimes M(1)$. The eigenspace of $\sigma\otimes \theta$ with eigenvalue $-1$ is equal to $$M=L_{sp(2n-2)}\left(-\frac{1}{2},0\right) \otimes M(1) ^- \oplus  L_{sp(2n-2)}\left(-\frac{1}{2},\lambda_1\right) \otimes M(1)^+.$$
  It follows from \cite[Theorem~4.4]{DM} that $M$ is an irreducible  $W_{-1}^{min}(sp(2n))$-module. Moreover, the conformal weight of $M$ is equal to $\frac{1}{2}$. Furthermore, $M(0)$ viewed as an $sp(2n-2)$-module is isomorphic to $L_{sp(2n-2)}(\lambda_1)$. It follows from Corollary \ref{isoM} that $$H_{\theta}(L_{sp(2n)}(-1,\bar{\Lambda}_2)) \cong  L_{sp(2n-2)}\left(-\frac{1}{2},0\right) \otimes M(1) ^- \oplus  L_{sp(2n-2)}\left(-\frac{1}{2},\lambda_1\right) \otimes M(1)^+.$$
  This completes the proof.
 \qed

 We next determine the decomposition of $H_{\theta}   (L_{sp(2n)}(-1,s\bar{\Lambda}_1) )$.
 \begin{theorem}\label{structureM2}
   As $L_{-\frac{1}{2}}(sp(2n-2)) \otimes M(1) ^+$-modules,
    $$H_{\theta}   (L_{sp(2n)}(-1,s\bar{\Lambda}_1) ) \cong   \mathcal S(n-1) \otimes M\left(1, \frac{s}{\sqrt{2 n}} \right), $$for $s \in {\Z}_{>0}$.
  \end{theorem}
 \begin{proof}
  By Proposition \ref{ZhuW}, $H_{\theta}   (L_{sp(2n)}(-1,s\bar{\Lambda}_1) ) (0)$ is an irreducible module of  $U(sp(2n-2))  \otimes {\C}[\bar\omega_2] $. It follows that $H_{\theta}   (L_{sp(2n)}(-1,s\bar{\Lambda}_1) )(0)$ is an irreducible module of  $U(sp(2n-2))$. By \cite[Theorem~6.3]{KW2}, $H_{\theta}(L_{sp(2n)}(-1,s\bar{\Lambda}_1) )(0)$ viewed as an $sp(2n-2)$-module is isomorphic to $L_{sp(2n-2)}(0)$. Moreover, by \cite[Theorem~7.1]{KW2}, the conformal weight of $H_{\theta}   (L_{sp(2n)}(-1,s\bar{\Lambda}_1) )$ is equal to $\frac{s^2}{4n}$.
  
  On the other hand, $\mathcal S(n-1) \otimes M\big(1, \frac{s}{\sqrt{2n}}\big)$ is an irreducible module of $\mathcal{S}(n-1)\otimes M(1)$. By \cite[Theorem~6.1]{DM}, $\mathcal S(n-1) \otimes M\big(1, \frac{s}{\sqrt{2n}}\big)$ is an irreducible $\mathcal{S}(n-1)\otimes M(1)^{\sigma\otimes \theta}$-module, and hence also an irreducible $W_{-1}^{min}(sp(2n))$-module. Moreover, the conformal weight of $\mathcal S(n-1) \otimes M\big(1, \frac{s}{\sqrt{2n}}\big)$ is equal to $\frac{s^2}{4n},$ and $\mathcal S(n-1) \otimes M\big(1, \frac{s}{\sqrt{2n}}\big)(0)$ viewed as an $sp(2n-2)$-module is isomorphic to $L_{sp(2n-2)}(0)$. Then the conclusion follows from Corollary \ref{isoM}.
 \end{proof}

\section{Semisimplicity of $KL_{-1}(sp(2n))$}\label{cn}
\def\theequation{4.\arabic{equation}}
\setcounter{equation}{0}
In this section, we assume that $n\geq 2$. We show that the category $\cC_{-1}(sp(2n))$ of grading-restricted generalized modules for the affine vertex operator algebra $L_{-1}(sp(2n))$ is semisimple, and hence the Kazhdan-Lusztig category $KL_{-1}(sp(2n))$ is the same as $\cC_{-1}(sp(2n))$, and admits a braided tensor category structure. By Theorem \ref{keysemi}, to prove that $\cC_{-1}(sp(2n))$ is semisimple, it is enough to prove that every highest weight $L_{-1}(sp(2n))$-module in $\cC_{-1}(sp(2n))$ is irreducible.

First we need the classification of irreducible $L_{-1}(sp(2n))$-modules in $\cC_{-1}(sp(2n))$. For $n \geq 3$, this is obtained in \cite{AP1,AP2} by using explicit formulas for singular vectors in $V^{-1}(sp(2n))$.  But for $n=2$, the singular vector approach is more technical, we have to use minimal $W$--algebras and quantum Hamilton reduction. We will present the detail in the proof of the following theorem:
\begin{theorem}\label{clasirr}
Let $\bar{\Lambda}_1, \cdots, \bar{\Lambda}_{n}$ be the fundamental weights of $sp(2n)$. Then
$$\{L_{sp(2n)}(-1, s\bar{\Lambda}_1)|s\in \Z_{\geq 0}\}\cup \{L_{sp(2n)}(-1, \bar{\Lambda}_2)\}$$
provides a complete list of irreducible $L_{-1}(sp(2n))$-modules in $\cC_{-1}(sp(2n))$.
\end{theorem}
\begin{proof}
In the case $n \ge 3$, the assertion follows from classification results obtained  in Corollary 4.1 of \cite{AP1} and Proposition 4 of \cite{AP2}.

Assume now that $n=2$. Using the decomposition of conformal embedding  $$L_{-1}(sp(2n)) \hookrightarrow L_{-1}(sl(2n))$$ from \cite{AP1}  ( see also  Theorem \ref{dec-mn} below), which also holds for $n=2$, we see that $$ L_{sp(4)}(-1, s\bar{\Lambda}_1), s\in \Z_{\geq 0}\; \text{ and } \; L_{sp(4)}(-1, \bar{\Lambda}_2)$$ are also   $L_{-1}(sp(4))$-modules.

Assume that $M$ is an ordinary $L_{-1}(sp(4))$-module. Then it has the form $$ M= L_{sp(4)}(-1, s\bar{\Lambda}_1 +   t\bar{\Lambda}_2),$$ for certain $  s, t\in \Z_{\geq 0}$. By Theorem \ref{exact}, $H_{\theta}(M)$ is an irreducible, ordinary,  highest weight
$W_{-1}^{min}(sp(4))$--module. This implies that  the top component of $H_{\theta}(M)$ is $sl(2)$--module with highest weight   $t \bar{\Lambda}_1$. Using the classification of $L_{-\frac{1}{2}} (sl(2))$--modules from \cite{AM1}, we know $t=0$ or $t=1$.

For $t=0$, we get $M= L_{sp(4)}(-1, s\bar{\Lambda}_1)$.
For $t=1$, the top space of $H_{\theta} (M)$ has conformal weight

 $$ \frac{(  s\bar{\Lambda}_1 + \bar{\Lambda}_2  + 2 \rho \vert    s\bar{\Lambda}_1 + \bar{\Lambda}_2 ) }{ 2 (-1 + h^{\vee}) }- \frac{1}{2} ( s\bar{\Lambda}_1 + \bar{\Lambda}_2 \vert \theta)  =  \frac{1}{2} + \Delta_s, $$
 where $\Delta_s = \frac{s^2 +2 s }{8}$.

Assume that the top space of $H_{\theta} (M)$ is isomorphic to the $sl(2)$--module with $sl(2)$ weight $\bar{\lambda}_1$.  Let $w$ be a highest weight vector of  $H_{\theta} (M)$. Then regarding $H_{\theta} (M)$ as an $L_{-\frac{1}{2}} (sl(2)) \otimes M(1)^+$--module, we see that
 $$(L_{-\frac{1}{2}} (sl(2)) \otimes M(1)^+)\cdot  w = L_{sl(2)}\left(-\frac{1}{2},\bar\lambda_1\right) \otimes W^1$$
 where $W^1$ is a highest $M(1)^+$--module whose top component has conformal weight
 $\Delta_s$, and $\bar{\lambda}_1$ is the fundamental weight of $sl(2)$.

If $s >0$, then $\Delta_s >0$.  First we notice the following fusion rules for  $L_{-\frac{1}{2}} (sl(2))$-modules (cf.  \cite[Section 5.2.1]{AKMPP2}):
  $$L_{sl(2)} \left(-\frac{1}{2},\bar\lambda_1\right) \boxtimes  L_{sl(2)} \left(-\frac{1}{2},\bar\lambda_1\right) =  L_{-\frac{1}{2}} (sl(2)).$$
Using fusion rules calculation from \cite{Abe3} (see Step 2 of \cite[Section 4.2]{Abe3}
  \protect\footnote{Strictly speaking Abe proved that if $N, L$ are irreducible $M(1)^+$--modules  and if  there is non-zero intertwining operator of type
   $\binom{L}{M(1)^- \   N }$, then  $(N, L)$ is one of the following pairs $(M(1)^{\pm}, M(1)^{\mp})$,  $(M(1, \theta) ^{\pm},  M(1, \theta) ^{\mp})$, $(M(1, \lambda), M(1, \mu)) \  (\lambda^2 = \mu ^2)$. But the  calculation   uses  the structure of the $A(M(1)^+)$--bimodule $A(M(1)^-)$, and doesn't depend on irreducibility of $N$ and $L$.  Therefore these arguments can be applied on $N= W^1$ and $L = W^0$ without assumption of irreducibility. }
   ), one concludes that:
  \begin{itemize}
  \item if $s=2$,  $ W^0 \subset M(1)^-  \times   W^1$, where $W^0$ is a highest weight module for $M(1)^+$ of conformal weight $0$;
  \item if $ s \ne 2$;  $ W^0 \subset M(1)^-  \times   W^1$, where $W^0$ is a highest weight module for $M(1)^+$ of conformal weight $\Delta_s$.
 \end{itemize}
This implies that  $(L_{sl(2)} \big(-\frac{1}{2},\bar\lambda_1\big) \otimes M(1)^- )  \cdot (L_{sl(2)} \big(-\frac{1}{2},\bar\lambda_1\big) \otimes W^1)$  contains an $L_{-\frac{1}{2}} (sl(2)) \otimes M(1)^+$-submodule of the form
  $  L_{-\frac{1}{2}} (sl(2)) \otimes  W^0$   having  lowest weight component of
  \begin{itemize}
  \item conformal weight  zero, if $s=2$; or
  \item  conformal weight $\Delta_s$  if $ s\ne 2$.
  \end{itemize}
This is a contradiction, since the top space of  $H_{\theta}(M)$ has conformal weight $\frac{1}{2} + \Delta_s$ which can not be neither zero or $\Delta_s$. Thus we have proved that $s=0$, so $M= L_{sp(2n)}(-1, \bar{\Lambda}_2)$.
\end{proof}


We next show that every highest weight $L_{-1}(sp(2n))$-module in  $\cC_{-1}(sp(2n))$ is irreducible. First, we show highest weight $L_{-1}(sp(2n))$-module of highest weight $-\Lambda_0+s\bar{\Lambda}_1$ is irreducible:
\begin{proposition}\label{irre3}
Suppose $M$ is a highest weight $L_{-1}(sp(2n))$-module in $\cC_{-1}(sp(2n))$ with highest weight $-\Lambda_0+s\bar{\Lambda}_1$ for $s\in \Z_{> 0}$, then $M$ is irreducible.
\end{proposition}
\begin{proof}
By \cite[Theorem~6.3]{KW2}, $H_{\theta}(M)$ is a highest weight module of $W_{-1}^{min}(sp(2n))$. Since $L_{-\frac{1}{2}}(sp(2n))\otimes M(1)^+$ is a subalgebra of $W_{-1}^{min}(sp(2n))$, $H_{\theta}(M)$ is also a $L_{-\frac{1}{2}}(sp(2n))\otimes M(1)^+$-module. Let $w^s$ be the highest weight vector of $H_{\theta}(M)$.
Since $$[J_3, x_0]=0$$ for any $x\in sp(2n-2)$, $J_3w^s$ is also a highest weight vector of $H_{\theta}(M)$. This implies that the $M(1) ^+$-submodule of $H_{\theta}(M)$ generated by $w^s$ is a highest weight module for $M(1)^+$. Moreover, by Theorem \ref{structureM2}, the conformal weight of $w^s$ with respect to $M(1)^+$ is $\frac{s^2}{4n}$. Therefore, by Theorem 3.8, the $M(1) ^+$-submodule of $H_{\theta}(M)$ generated by $w^s$ must be isomorphic to $M(1, \frac{s}{\sqrt{2n}})$.

Since $-\frac{1}{2}$ is an admissible level of $sp(2n-2)$ (cf. \cite{KW1}), it follows from  Theorem \ref{structureM2} that $L_{sp(2n-2)}(-\frac{1}{2},0) \otimes M(1) ^+$-submodule of $H_{\theta}(M)$ generated by $w^s$ is isomorphic to $L_{sp(2n-2)}\big(-\frac{1}{2},0\big) \otimes M\big(1, \frac{s}{\sqrt{2 n}} \big)$. By Proposition \ref{fusionM2}, $$M(1)^- \boxtimes M\left(1, \frac{s}{\sqrt{2n}}\right)= M\left(1, \frac{s}{\sqrt{2 n}}\right).$$
This implies that
$$\bigg\{u_nv|u\in L_{sp(2n-2)}\left(-\frac{1}{2},\lambda_1\right)  \otimes M(1)^-, v\in L_{sp(2n-2)}\left(-\frac{1}{2},0\right) \otimes M\left(1, \frac{s}{\sqrt{2 n}}\right), n\in \Z\bigg\}$$
must be isomorphic to $L_{sp(2n-2)}\big(-\frac{1}{2},\lambda_1\big) \otimes M\big(1, \frac{s}{\sqrt{2 n}}\big )$. Since  $H_{\theta}(M)$ is generated by $w^s$, $H_{\theta}(M)$ has the following decomposition as an $L_{sp(2n-2)}\big(-\frac{1}{2},0\big) \otimes M(1) ^+$-module:
 $$H_{\theta}(M)\cong L_{sp(2n-2)}\left(-\frac{1}{2},0\right) \otimes M\left(1, \frac{s}{\sqrt{2 n}}\right)\oplus L_{sp(2n-2)}\left(-\frac{1}{2},\lambda_1\right) \otimes M\left(1, \frac{s}{\sqrt{2 n}}\right).$$
 Therefore, it follows from Theorem \ref{structureM2} that $H_{\theta}(M)$ is an irreducible $W_{-1}^{min}(sp(2n))$-module. Hence, $M$ is an irreducible $L_{-1}(sp(2n))$-module.
\end{proof}

We next show that every highest weight $L_{-1}(sp(2n))$-module of highest weight $-\Lambda_0$ is irreducible:
\begin{proposition}\label{irre1}
Suppose $M$ is a highest weight $L_{-1}(sp(2n))$-module in  $\cC_{-1}(sp(2n))$ with highest weight $-\Lambda_0$, then $M$ is irreducible.
\end{proposition}

\begin{proof}
Let $w$ be the highest weight vector of $M$. Since
$$[L(-1), x(n)]=-nx(n-1)$$
for any $x\in sp(2n)$, $L(-1)w$ is also a highest weight vector for $L_{-1}(sp(2n))$. Since the weight one subspace of $M$ does not contain a singular vector for $L_{-1}(sp(2n))$,  $L(-1)w=0$. Therefore,  $w$ is a vacuum-like vector. By Proposition 4.7.7 of \cite{LL}, $M$ must be isomorphic to $L_{-1}(sp(2n))$. Thus, $M$ is irreducible.
\end{proof}

Finally, we show that every highest weight $L_{-1}(sp(2n))$-module of highest weight $-\Lambda_0+\bar{\Lambda}_2$ is irreducible.
\begin{proposition}\label{irre2}
Suppose $M$ is a highest weight $L_{-1}(sp(2n))$-module in $\cC_{-1}(sp(2n))$ with highest weight $-\Lambda_0+\bar{\Lambda}_2$, then $M$ is irreducible.
\end{proposition}
\begin{proof} By \cite[Theorem~6.3]{KW2}, $H_{\theta}(M)$ is a highest weight module of $W_{-1}^{min}(sp(2n))$, and $H_{\theta}(M)$ can be viewed as an $L_{-\frac{1}{2}}(sp(2n))\otimes M(1)^+$-module.
Let $w$ be the highest weight vector of $H_{\theta}(M)$. By the similar argument as in Proposition \ref{irre3}, the $M(1) ^+$-submodule $\langle w\rangle$ of $H_{\theta}(M)$ generated by $w$ is a highest weight module for $M(1) ^+$. Moreover, by Theorem \ref{structureM}, the conformal weight of $w^s$ with respect to  $M(1) ^+$ is $0$. Therefore, by Theorem 3.8, the $M(1) ^+$-submodule $\langle w\rangle$ is isomorphic to  $M(1) ^+$ or there is an exact sequence
\begin{equation*}
    0 \rightarrow M(1) ^- \rightarrow \langle w\rangle \rightarrow M(1) ^+ \rightarrow 0.
    \end{equation*}

If $\langle w\rangle$ is isomorphic to  $M(1) ^+$, by the similar argument as in Proposition \ref{irre3}, $H_{\theta}(M)$ has the following decomposition as a $L_{-\frac{1}{2}}(sp(2n))\otimes M(1)^+$-module:
 $$H_{\theta}(M)\cong  L_{sp(2n-2)}\left(-\frac{1}{2},0\right) \otimes M(1) ^- \oplus  L_{sp(2n-2)}\left(-\frac{1}{2},\lambda_1\right) \otimes M(1) ^+.$$
Therefore, it follows from Theorem \ref{structureM} that $H_{\theta}(M)$ is an irreducible $W_{-1}^{min}(sp(2n))$-module. Hence, $M$ is an irreducible $L_{-1}(sp(2n))$-module.

We next consider the case that there is an exact sequence
\begin{equation*}
    0 \rightarrow M(1) ^- \rightarrow \langle w\rangle \rightarrow M(1) ^+ \rightarrow 0.
    \end{equation*} The $L_{sp(2n-2)}(-\frac{1}{2},0) \otimes M(1) ^+$-submodule of $H_{\theta}(M)$ generated by $w$ contains a submodule isomorphic to $L_{sp(2n-2)}\big(-\frac{1}{2},\lambda_1\big) \otimes M(1) ^-$. By Proposition \ref{fusionM1}, $$M(1)^- \boxtimes M(1)^- = M(1)^+.$$
    This implies that
$$\bigg\{u_nv\bigg|u, v\in L_{sp(2n-2)}\left(-\frac{1}{2},\lambda_1\right)  \otimes M(1)^-, n\in \Z\bigg\}$$
must be isomorphic to $L_{sp(2n-2)}\big(-\frac{1}{2},0\big) \otimes M(1) ^+$. Therefore, $H_{\theta}(M)$ contains a $W_{-1}^{min}(sp(2n))$-submodule isomorphic to $W_{-1}^{min}(sp(2n))$. However, $H_{\theta}(M)$ is a highest weight module for $W_{-1}^{min}(sp(2n))$ generated by $w$ and the conformal weight of  $H_{\theta}(M)$ is $\frac{1}{2}$. We then  obtain a contradiction.
\end{proof}

Now we are ready to prove the main result in this section.
\begin{theorem}\label{semitypeC}
The category $\cC_{-1}(sp(2n))$ is semisimple.
\end{theorem}
\begin{proof}
By Theorem \ref{keysemi}, it is enough to prove that every highest weight $L_{-1}(sp(2n))$-module in $\cC_{-1}(sp(2n))$ is irreducible. This follows from Propositions \ref{irre3}, \ref{irre1}, \ref{irre2}.
\end{proof}

By Theorem \ref{semitypeC}, the category $KL_{-1}(sp(2n))$ is the same as the category $\cC_{-1}(sp(2n))$, and hence it is a semisimple category consisting of direct sums of irreducible modules $L_{sp(2n)}(-1, s\bar{\Lambda}_1)$, $s\in \Z_{\geq 0}$ or $L_{sp(2n)}(-1, \bar{\Lambda}_2)$. Moreover, by Theorem \ref{tensorCY}, we have
\begin{theorem}\label{tencn}
The category $KL_{-1}(sp(2n))$ is a braided tensor category.
\end{theorem}

\section{Schur-Weyl duality between $L_{-1}(sp(2n))$ and $M(1)^+$}
In this section, we will show that there exists a Schur-Weyl duality between vertex operator algebras $L_{-1}(sp(2n))$ and $M(1)^+$ for $n \geq 2$.

First, we recall the Schur-Weyl duality between vertex operator algebras $L_{-1}(sl(2n))$ and $M(1)$ obtained in \cite{AP1}. Consider the Weyl vertex algebra $\mathcal{S}(2n)$ (cf. Sec. \ref{weyl}). Set $$ H =  \sum_{i=1}^{2n}  a_i ^+\left(-\frac{1}{2}\right) a_i ^-\left(-\frac{1}{2}\right)\1, $$
and $\h_1=\C H$. Let $h = \frac{1}{\sqrt{-2n}} H$. Define a  bilinear form $(,)$ on $\h_1$ such that
$(h, h)=1$. Then it is known that the subalgebra $\langle H\rangle$ of $\mathcal{S}(2n)$ generated by $H$ is isomorphic to the Heisenberg vertex operator algebra $M_{\h_1}(1)$ (cf. Section 6 of \cite{AP2}). For $\beta\in \h_1$, we use  $M_{\h_1}(1,\beta)$ to denote the highest weight $M_{\h_1}(1)$-modules such that $h(0)$ acts as $\beta \cdot{\id}$.

It is proved in \cite[Theorem~3.2]{AP1} that $L_{-1}(sl(2n))$ and $M_{\h_1}(1)$ form a commutant pair in $\mathcal{S}(2n)$:
\begin{theorem}\label{swAM}
The Weyl vertex algebra $\mathcal{S}(2n)$ viewed as an $L_{-1}(sl(2n))\otimes M_{\h_1}(1)$-module has the following decomposition:
$$
    \mathcal{S}(2n) = \bigoplus_{s =0}^\infty L_{sl(2n)}(-1, s\dot{\Lambda}_1) \otimes M_{\h_1}\left(1,\frac{s}{\sqrt{-2n}}\right) \oplus \bigoplus_{s = 1}^\infty L_{sl(2n)}(-1, s\dot{\Lambda}_{2n-1}) \otimes M_{\h_1}\left(1,-\frac{s}{\sqrt{-2n}}\right),$$
where $\dot{\Lambda}_1, \dots, \dot{\Lambda}_{2n-1}$ are the fundamental weights of the Lie algebra $sl(2n)$.
\end{theorem}

Following \cite{AP2}, we define an involution $\sigma$ of $\mathcal{S}(2n)$ such that
\begin{align*}
&\sigma (a_i^+) = a_{2n+1-i} ^-, \quad  \sigma (a_{2n+1-i} ^+)=-a_i^-, \\
&\sigma (a_i^-) = -a_{2n+1-i} ^+, \quad  \sigma (a_{2n+1-i} ^-)=a_i^+ 
\end{align*}
for $i=1, \dots, n$. Note that if $S$ be a subspace of $\mathcal{S}(2n)$  which is  $\sigma$--invariant, then
$S = S^+\oplus S^-$, where $S^{\pm} = \{v \in S \vert \ \sigma(v) = \pm v\}$. The vertex subalgebras $L_{-1} (sl(2n))$ and $M_{\h_1}(1)$ are both $\sigma$-invariant and
$\sigma(H)=-H$. Moreover, it was proved in \cite[Section~8]{AP2} that 
\begin{equation}\label{fixptsigma}
    (L_{-1} (sl(2n)))^+ = L_{-1} (sp(2n)),\quad  (L_{-1} (sl(2n)))^- = L_{sp(2n)}(-1, \bar{\Lambda}_2).
\end{equation}


Now we prove the Schur-Weyl duality between $L_{-1}(sp(2n))$ and $M_{\h_1}(1)^+$:
\begin{theorem}\label{dec-mn}
The $\Z_2$-orbifold subalgebra $\mathcal{S}(2n)^+$ has the following decomposition:
\begin{align*}
\mathcal{S}(2n)^+ \cong &  L_{-1}(sp(2n))\otimes M_{\h_1}(1)^+\oplus L_{sp(2n)}(-1, \bar{\Lambda}_2)\otimes M_{\h_1}(1)^-\\
  &\quad \oplus   \bigoplus_{s = 1}^\infty L_{sp(2n)}(-1, s\bar{\Lambda}_1) \otimes M_{\h_1}\left(1,\frac{s}{\sqrt{-2n}}\right),
\end{align*}
as an $L_{-1}(sp(2n))\otimes M_{\h_1}(1)^+$-module.
\end{theorem}
\begin{proof}
By Theorem \ref{swAM}, $\mathcal{S}(2n)$ viewed as an $L_{-1}(sl(2n))\otimes M_{\h_1}(1)$-module has the following decomposition:
$$
    \mathcal{S}(2n) = \bigoplus_{s =0}^\infty L_{sl(2n)}(-1, s\dot{\Lambda}_1) \otimes M_{\h_1}\left(1,\frac{s}{\sqrt{-2n}}\right) \oplus \bigoplus_{s = 1}^\infty L_{sl(2n)}(-1, s\dot{\Lambda}_{2n-1}) \otimes M_{\h_1}\left(1,-\frac{s}{\sqrt{-2n}}\right).$$
Also, we know from \eqref{fixptsigma} that $L_{-1} (sl(2n))\otimes M_{\h_1}(1)$ is $\sigma$-invariant, and $$(L_{-1} (sl(2n))\otimes M_{\h_1}(1))^+=L_{-1}(sp(2n))\otimes M_{\h_1}(1)^+\oplus L_{sp(2n)}(-1, \bar{\Lambda}_2)\otimes M_{\h_1}(1)^-.$$

For $s \in {\Z}\setminus \{0\}$, it follows from \cite[Lemma~2.1]{AP1} that
$$\sigma\left(L_{sl(2n)}(-1, s\dot{\Lambda}_1) \otimes M_{\h_1}\left(1,\frac{s}{\sqrt{-2n}}\right)\right)=L_{sl(2n)}(-1, s\dot{\Lambda}_{2n-1}) \otimes M_{\h_1}\left(1,-\frac{s}{\sqrt{-2n}}\right).$$ Then it follows from \cite[Theorem~6.1]{DM} that $L_{sl(2n)}(-1, s\dot{\Lambda}_1) \otimes M_{\h_1}\big(1,\frac{s}{\sqrt{-2n}}\big)$ and\\ $L_{sl(2n)}(-1, s\dot{\Lambda}_{2n-1}) \otimes M_{\h_1}\big(1,-\frac{s}{\sqrt{-2n}}\big)$ are irreducible $L_{-1}(sp(2n))\otimes M_{\h_1}(1)^+$-modules. From  \cite[Proposition 7]{AP2}, as $L_{-1}(sp(2n))$--modules
$$ L_{sl(2n)}(-1, s\dot{\Lambda}_1) \cong L_{sl(2n)}(-1, s\dot{\Lambda}_{2n-1}) \cong L_{sp(2n)}(-1, s\bar{\Lambda}_1) $$
for $s \in \mathbb{Z}_{\geq 1}$.
This implies that  as $L_{-1}(sp(2n))\otimes M_{\h_1}(1)^+$-modules:
\begin{align*}
&\left(L_{sl(2n)}(-1, s\dot{\Lambda}_1) \otimes M_{\h_1}\left(1,\frac{s}{\sqrt{-2n}}\right) \bigoplus  L_{sl(2n)}(-1, s\dot{\Lambda}_{2n-1}) \otimes M_{\h_1}\left(1,-\frac{s}{\sqrt{-2n}}\right)\right)^{+} \\
&\cong  L_{sp(2n)}(-1, s\bar{\Lambda}_1) \otimes M_{\h_1}\left(1,\frac{s}{\sqrt{-2n}}\right).
\end{align*}
The statement then follows.
\end{proof}

Since $KL_{-1}(sp(2n))$ consists of direct sums of $L_{sp(2n)}(-1, \bar{\Lambda}_2)$ and $L_{sp(2n)}(-1, s\bar{\Lambda}_1)$ for $s \in \Z_{\geq 0}$, by \cite[Theorem~1.1]{M3} and Theorem \ref{dec-mn}, we have
\begin{corollary}
    As braided tensor categories, $KL_{-1}(sp(2n))$ is braided reversed equivalent to the full subcategory of $\cC_1(M_{\h_1}(1)^+)$ generated by $M_{\h_1}\big(1, \frac{s}{\sqrt{-2n}}\big)$ for $s \in \Z_{\geq 1}$ and $M_{\h_1}(1)^\pm$ under the functor that sends
    \[
    L_{sp(2n)}(-1, s\bar{\Lambda}_1) \mapsto M_{\h_1}\left(1, \frac{s}{\sqrt{-2n}}\right),\;\; L_{sp(2n)}(-1, \bar{\Lambda}_2) \mapsto M_{\h_1}(1)^-.
    \]
    In particular, $KL_{-1}(sp(2n))$ is rigid, with the following fusion product decompositions:
    \begin{align*}
&        L_{sp(2n)}(-1, \bar{\Lambda}_2) \boxtimes L_{sp(2n)}(-1, \bar{\Lambda}_2) \cong L_{-1}(sp(2n)),\\
&  L_{sp(2n)}(-1, \bar{\Lambda}_2) \boxtimes L_{sp(2n)}(-1, s\bar{\Lambda}_1) \cong L_{sp(2n}(-1, s\bar{\Lambda}_1),\\
& L_{sp(2n)}(-1, s\bar{\Lambda}_1) \boxtimes L_{sp(2n)}(-1, t\bar{\Lambda}_1) \cong L_{sp(2n)}(-1, (s+t)\bar{\Lambda}_1) \oplus L_{sp(2n)}(-1, (s-t)\bar{\Lambda}_1)
    \end{align*}
    for $s, t \in \Z_{\geq 1}$ and $s \geq t$.
\end{corollary}

\appendix
\section{Proof of Lemma \ref{keylemma}}\label{AppA}

Let $\gamma=\sqrt{2}\alpha$ and $L= \Z \gamma$, and let $V_L$ be the lattice vertex operator algebra associated to $L$.  The conformal element of $V_L$ is $\frac{1}{2}\alpha(-1)^2$, and the subalgebra of $V_L$ generated by the conformal element is isomorphic to the simple Virasoro vertex operator algebra of central charge $1$, denoted by $L(1,0)$. Moreover, the weight one subspace of $V_L$ is isomorphic to the Lie algebra  $\mathfrak{sl}_2 = \mbox{span} \{E, F, H \}$. Therefore,  $V_L$ can be viewed as a module for $\mathfrak{sl}_2 \oplus L(1,0)$ (see \cite[Corollary~2.4]{DG}). Furthermore, in \cite[Equation~(2.2)]{DG} it is proved that $V_L$ viewed as an $\mathfrak{sl}_2 \oplus L(1,0)$-module has the following decomposition:
$$ V_{L}  = \bigoplus_{m=0} ^{\infty} \rho_{2m} \otimes L  (1, m^2),$$
where $\rho_n$ is the  $\mathfrak{sl}_2$-module of dimension $n+1$, and $L\left(1, h\right)$ denotes the highest weight irreducible $L\left(1, 0\right)$-module of highest weight $h$. Denote the highest weight vector of the $\mathfrak{sl}_2 \oplus L(1, 0)$-module $\rho_{2m} \otimes L  (1, m^2)$ by $e^{m \gamma}$.

The modules $V_L$, $V_{L+\gamma/2}$ are irreducible $V_L$-modules (cf. \cite{FLM}). In particular, $V_{L+\gamma/2}$ can also be viewed as an $\mathfrak{sl}_2 \oplus L(1,0)$-module with the following decomposition:
$$ V_{L+\gamma/2}  = \bigoplus_{m=0} ^{\infty} \rho_{2m+1} \otimes L  \left(1, \frac{(2m+1)^2}{4}\right)$$
(\cite[Equation~19]{Mi}).

Since $V_L$ contains a subalgebra isomorphic to $M(1)$, $V_L$ and $V_{L+\gamma/2}$ can be regarded $M(1)$-modules. The $M(1)$-submodule of $V_L\oplus  V_{L+\gamma/2}$ generated by $e^{m \gamma /2}$ is isomorphic to $M\big(1, \frac{m}{\sqrt{2}}\alpha\big)$ for each $m\in \Z_{\geq 0}$. By a direct calculation, we have $F^k e^{ m \gamma/2 +k \gamma}\in M\big(1, \frac{m}{\sqrt{2}}\alpha\big)$. Moreover,
 \begin{align}\label{hwv}
 v_m ^{(k)} := F^k e^{ m \gamma/2 +k \gamma}
   \end{align}
   is a highest weight vector for the Virasoro algebra $L(1,0)$. Combining this fact with \cite[Equation~(3.2)]{Abe3}, $M\left(1, \frac{m}{\sqrt{2}}\alpha\right)$ has the following decomposition:
 \begin{align*}
 M\left(1, \frac{m}{\sqrt{2}}\alpha\right) = \bigoplus_{k=0} ^{\infty}   L\left(1, \left(\frac{m}{2}+k\right)^2\right)  = \bigoplus_{k=0} ^{\infty}   L(1,0 ) v_m ^{(k)}
 \end{align*}
 as an $L(1,0)$-module.

 Let $Y(\cdot,z)$ be the vertex operator map on the $V_L$-module $V_L\oplus  V_{L+\gamma/2}$. For any $u\in V_L$ and $n \in \Z$, let $u_n = {\rm Res}_z z^{n}Y(u,z)$. 
   Then  we have the following result:
   \begin{lemma}\label{basis0}
   There exists a nonzero number $\sigma$ such that
     \begin{align}\label{const}
     E^i(E^2J)_{(-2m-4k-1)}E^jv_m ^{(k)}=\sigma e^{m \gamma/2 +(k+2) \gamma},
     \end{align}
   holds  for any $i,j\in \Z_{\geq 0}$ such that $i+j=k$.
   \end{lemma}
   \pf
   First, by the definition of $J$, there exists a nonzero number $\sigma_1$ such that
     $E^2J=\sigma_1e^{2\gamma}$. And by equation (\ref{hwv}), there exists a nonzero number $\sigma_2$ such that
     \begin{align}\label{nonzero1}
     E^kv_m ^{(k)}=\sigma_2e^{m \gamma/2 +k \gamma}.
      \end{align}
      Thus, there exists a nonzero number $\sigma$ such that
     \begin{align*}
     (E^2J)_{(-2m-4k-1)}E^kv_m ^{(k)}=\sigma e^{m \gamma/2 +(k+2) \gamma}.
     \end{align*}
     Since $E^3J=0$, we have $[E, (E^2J)_{(-2m-4k-1)}]=0$. This implies that
      \begin{align*}
     E^i(E^2J)_{(-2m-4k-1)}E^jv_m ^{(k)}=(E^2J)_{(-2m-4k-1)}E^kv_m ^{(k)}=\sigma e^{m \gamma/2 +(k+2) \gamma}.
     \end{align*}
     This completes the proof.
     \qed

Using the formula (\ref{const}), we have the following result.

   \begin{lemma}\label{basis1}
 For any $i,j\in \Z_{\geq 0}$ such that $i+j=k+1$ and $i\geq 1$,
 \begin{align*}
     E^i(EJ)_{(-2m-4k-1)}E^jv_m ^{(k)}=i\sigma e^{m \gamma/2 +(k+2) \gamma}.
     \end{align*}
 \end{lemma}
 \pf We prove the statement by induction on $i$. We first consider the case that $i=1$. Since $E^{k+1}v_m ^{(k)}=0$, it follows from Lemma \ref{basis0} that \begin{align*}
     E(EJ)_{(-2m-4k-1)}E^{k}v_m ^{(k)}=(E^2J)_{(-2m-4k-1)}E^{k}v_m ^{(k)}=\sigma e^{m \gamma/2 +(k+2) \gamma}.
     \end{align*}
     We now assume that the statement holds for $i$. Then it follows from Lemma \ref{basis0} that
     \begin{align*}
     &E^{i+1}(EJ)_{(-2m-4k-1)}E^{j-1}v_m ^{(k)}\\
     &=E^{i}(EJ)_{(-2m-4k-1)}E^{j}v_m ^{(k)}+E^{i}(E^2J)_{(-2m-4k-1)}E^{j-1}v_m ^{(k)}\\
     &=(i+1)\sigma e^{m \gamma/2 +(k+2) \gamma}.
     \end{align*}
 This completes the proof.
 \qed

 As a consequence, we have the following result:
  \begin{lemma}\label{basis2}
 For any $i,j\in \Z_{\geq 0}$ such that $i+j=k+2$ and $i\geq 2$,
 \begin{align*}
     E^iJ_{(-2m-4k-1)}E^jv_m ^{(k)}=\frac{i(i-1)}{2}\sigma e^{m \gamma/2 +(k+2) \gamma}.
     \end{align*}
 \end{lemma}
 \pf We prove the statement by induction on $i$. We first consider the case that $i=2$. Since $E^{k+1}v_m ^{(k)}=0$, it follows from Lemma \ref{basis1} that \begin{align*}
     &E^2J_{(-2m-4k-1)}E^{k}v_m ^{(k)}\\
     &=E(EJ)_{(-2m-4k-1)}E^{k}v_m ^{(k)}\\
     &=\sigma e^{m \gamma/2 +(k+2) \gamma}.
     \end{align*}
     We now assume that the statement holds for $i$. Then it follows from Lemma \ref{basis0} that
     \begin{align*}
     &E^{i+1}J_{(-2m-4k-1)}E^{j-1}v_m ^{(k)}\\
     &=E^{i}J_{(-2m-4k-1)}E^{j}v_m ^{(k)}+E^{i}(EJ)_{(-2m-4k-1)}E^{j-1}v_m ^{(k)}\\
     &=\left(\frac{i(i-1)}{2}+i\right)\sigma e^{m \gamma/2 +(k+2) \gamma}\\
     &=\frac{i(i+1)}{2}\sigma e^{m \gamma/2 +(k+2) \gamma}.
     \end{align*}
 This completes the proof.
 \qed

Now we can give the proof of Lemma \ref{keylemma}:
\begin{proof}
From \cite{Mi} (cf. \cite{M1}), the semisimple subcategory of $L(1,0)$-modules generated by $L(1, \frac{m^2}{4})$ is a vertex tensor category with the fusion rules
 \begin{align*}
 &L\left(1, \frac{n^2}{4}\right) \boxtimes L\left(1, \frac{m^2}{4}\right)\\
  &=  L\left(1, \frac{(n+m)^2}{4}\right) \oplus  L\left(1, \frac{(n+m-2)^2}{4}\right) \oplus \cdots \oplus  L\left(1, \frac{(n-m)^2}{4}\right).
 \end{align*}
 In particular,
 \begin{align*}
 &L(1,4) \boxtimes L\left(1, \left(\frac{m}{2}+k\right)^2\right)\\
 &= L\left(1, \left(\frac{m}{2}+k+2\right)^2\right) \oplus  L\left(1, \left(\frac{m}{2}+k+1\right)^2\right) \oplus \cdots \oplus  L\left(1, \left(\frac{m}{2}+k-2\right)^2\right).
\end{align*}
By calculating the conformal weights, there exists a  constant $C$ such that
 \begin{equation*}
 J_{(-2m-4k-1)}v^{(k)}_m = Cv^{(k+2)}_m + w,
 \end{equation*}
 where $w \in  \bigoplus_{i= 0}^{k+1} L\left(1, \left(\frac{m}{2}+i\right)^2\right)$.
Since $E^{k+1}v^{(k)}_m=0$, it follows that $E^{k+2}w= 0$. Furthermore, by Lemma \ref{basis2}, $E^{k+2}J_{(-2m-4k-1)}v^{(k)}_m\neq 0$.  This implies that $C\neq 0$.
\end{proof}
 
\section{Proof of Proposition \ref{prop:genofU}}\label{appB}
Let $V$ be the subalgebra of $U$ generated by the following elements:
    \begin{align*}
    &x(-1)\1\otimes \1,~~~x\in sp(2n-2),\\
    &\omega_1\otimes \1+\1\otimes \omega_2,\\
    &y\otimes h(-1)\1,~~~y\in L_{sp(2n-2)}(\lambda_1),
    \end{align*}
   where $\omega_1$ and $\omega_2$ denote the Virasoro elements of $L_{sp(2n-2)}(-\frac{1}{2},0)$ and $M(1) ^+ $, respectively. It is known that $L_{sp(2n-2)}(\lambda_1)$ is the vector representation of $sp(2n-2)$ (cf. Proposition 13.24 of \cite{Ca}). It follows that $$ L_{sp(2n-2)}(\lambda_1)={\rm span}\{a_i^\pm|1\leq i\leq n-1\}. $$ Therefore,  $V$ is the subalgebra of $U$ generated by the following elements:
    \begin{align*}
    &x(-1)\1\otimes \1,~~~x\in sp(2n-2),\\
    &\omega_1\otimes \1+\1\otimes \omega_2,\\
    &a_i^\pm\otimes h(-1)\1,~~~1\leq i\leq n-1.
    \end{align*} 

In the following, we will show that $U=V$. First note the vertex operator algebra $M(1)^+$ is generated by its Virasoro element $\omega_2$
and $J=h(-1)^4\1-2h(-3)h(-1)\1+\frac{3}{2}h(-2)^2\1.$ (\cite[Theorem~2.7]{DG}, cf. Sec \ref{sec:span}). Next, we will prove $1 \otimes J \in V$.

  To show that $U=V$, the key point is to prove that $1\otimes J\in V$. We first show that $\1\otimes h(-1)^4\1\in V$.
  \begin{lemma}\label{geners1}
  $\1\otimes h(-1)^4\1\in V$.
  \end{lemma}
  \pf We first show that $\1\otimes h(-1)^2\1\in V$. Let $Y(\cdot ,z)$ be the vertex operator map of $U$, and $u,v$ be elements in $U$. We let $Y(u, z)v=\sum_{n\in\Z}u_{(n)}vz^{-n-1}$. By the definition of $\mathcal{S}(n-1)$, we have
  $$(a_1^+)_{(0)}a_1^-=a_1^+(\frac{1}{2})a_1^-(-\frac{1}{2})\1=[a_1^+(\frac{1}{2}),a_1^-(-\frac{1}{2})]\1=\1.$$
 Therefore, by the definition of tensor product of vertex operator algebras (cf. Section 2.5 of \cite{FHL}), we have
 \begin{align*}
 &(a_1^+\otimes h)_{(0)}(a_1^-\otimes h)\\
 &=(a_1^+)_{(0)}a_1^-\otimes h_{(-1)}h+(a_1^+)_{(-1)}a_1^-\otimes h_{(0)}h\\
 &=\1\otimes h_{(-1)}^2\1.
 \end{align*}
 Hence, $\1\otimes h(-1)^2\1\in V$.

 We next show that $\1\otimes h(-1)^4\1\in V$. First,
  \begin{align*}
 &(a_1^-\otimes h)_{(-1)}(\1\otimes h(-1)^2\1)\\
 &=(a_1^-)_{(-2)}\1\otimes h(0)h(-1)^2\1+(a_1^-)_{(-1)}\1\otimes h(-1)^3\1+(a_1^-)_{(0)}\1\otimes h(-2)h(-1)^2\1\\
 &=(a_1^-)_{(-1)}\1\otimes h(-1)^3\1.
 \end{align*}
 Hence, we have $a_1^-\otimes h(-1)^3\1\in V$. Since
 \begin{align*}
 &(a_1^+\otimes h)_{(0)}(a_1^-\otimes h(-1)^3\1)\\
 &=(a_1^+)_{(0)}a_1^-\otimes h(-1)^4\1+(a_1^+)_{(-1)}a_1^-\otimes h(0)h(-1)^3\1\\
 &=\1\otimes h(-1)^4\1,
 \end{align*}
 we have $\1\otimes h(-1)^4\1\in V$. This completes the proof.
  \qed

  We next show that $\1\otimes h(-3)h(-1)\1\in V$.
  \begin{lemma}\label{geners2}
   $\1\otimes h(-3)h(-1)\1\in V$.
  \end{lemma}
  \pf First, we have
 \begin{align*}
 &(a_1^+\otimes h)_{(-2)}(a_1^-\otimes h)\\
 &=(a_1^+)_{(0)}a_1^-\otimes h(-3)h+(a_1^+)_{(-1)}a_1^-\otimes h(-2)h+(a_1^+)_{(-2)}a_1^-\otimes h(-1)^2\1\\
 &=\1\otimes h(-3)h+(a_1^+)_{(-1)}a_1^-\otimes h(-2)h+(a_1^+)_{(-2)}a_1^-\otimes h(-1)^2\1.
 \end{align*}
 Therefore, it is enough to prove that $(a_1^+)_{(-1)}a_1^-\otimes h(-2)h\in V$ and $(a_1^+)_{(-2)}a_1^-\otimes h(-1)^2\1\in V$.

 Since $(a_1^+)_{(-2)}a_1^-\in \mathcal{S}(n-1)^{\bar 0}$, we have $(a_1^+)_{(-2)}a_1^-\otimes \1\in V$. And in the proof of  Lemma \ref{geners1}, we have proved that $\1\otimes h(-1)^2\1\in V$. Moreover,
 \begin{align*}
 &((a_1^+)_{(-2)}a_1^-\otimes \1)_{(-1)}(\1\otimes h(-1)^2\1)=(a_1^+)_{(-2)}a_1^-\otimes h(-1)^2\1.
 \end{align*}
 Thus, $(a_1^+)_{(-2)}a_1^-\otimes h(-1)^2\1\in V$.

 We next show that $\1\otimes h(-2)h(-1)\1\in V$. First, \begin{align*}
 &((a_1^+)_{(-1)}a_1^-\otimes \1)_{(-1)}(\1\otimes h(-1)^2\1)=(a_1^+)_{(-1)}a_1^-\otimes h(-1)^2\1.
 \end{align*}
 it follows that $(a_1^+)_{(-1)}a_1^-\otimes h(-1)^2\1\in V$.
 Moreover, \begin{align*}
 &(a_1^+\otimes h)_{(-1)}(a_1^-\otimes h)\\
 &=(a_1^+)_{(0)}a_1^-\otimes h(-2)h+(a_1^+)_{(-1)}a_1^-\otimes h(-1)^2\1\\
 &=\1\otimes h(-2)h+(a_1^+)_{(-1)}a_1^-\otimes h(-1)^2\1.
 \end{align*}
   This implies that $\1\otimes h(-2)h\in V$.

   Since \begin{align*}
 &((a_1^+)_{(-1)}a_1^-\otimes \1)_{(-1)}(\1\otimes h(-2)h(-1)\1)=(a_1^+)_{(-1)}a_1^-\otimes h(-2)h(-1)\1,
 \end{align*}
 we have $(a_1^+)_{(-1)}a_1^-\otimes h(-2)h(-1)\1\in V$. This completes the proof.
  \qed

 Finally, we  show that $\1\otimes h(-2)^2\1\in V$.
  \begin{lemma}\label{geners3}
$\1\otimes h(-2)^2\1\in V$.
  \end{lemma}
  \pf First, we show that $a_1^-\otimes h(-2)\1\in V$. Note that
  \begin{align*}
 &(\1\otimes \omega_2)_{(0)}(a_1^-\otimes h(-1)\1)\\
 &=(\1)_{(-1)}a_1^-\otimes (\omega_2)_{(0)}h(-1)\1\\
 &=a_1^-\otimes h(-2)\1.
 \end{align*}
 Thus, $a_1^-\otimes h(-2)\1\in V$.

 We now show that $\1\otimes h(-2)^2\1\in V$. Note that
 \begin{align*}
 &(a_1^+\otimes h)_{(-1)}(a_1^-\otimes h(-2)\1)\\
 &=(a_1^+)_{(0)}a_1^-\otimes h(-2)^2\1+(a_1^+)_{(-1)}a_1^-\otimes h(-1)h(-2)\1\\
 &=\1\otimes h(-2)^2\1+(a_1^+)_{(-1)}a_1^-\otimes h(-2)h(-1)\1.
 \end{align*}
 This implies that $\1\otimes h(-2)^2\1\in V$.
  \qed

Now we can prove Proposition \ref{prop:genofU}: By Lemma \ref{geners1}, \ref{geners2}, we have $1\otimes J\in V$. Because $M(1)^+$ is generated by $\omega_2$ and $J$, $L_{sp(2n-2)}\big(-\frac{1}{2},0\big) \otimes M(1) ^+$ is contained in $V$. Since $y\otimes h(-1)\1\in V$ for any $y\in L_{sp(2n-2)}(\lambda_1)$,  it follows that $L_{sp(2n-2)}\big(-\frac{1}{2},\lambda_1\big)  \otimes M(1)^-$ is contained in $V$ as well. Consequently $U=V$.
  \qed

\subsection*{Acknowledgement}    
 D. Adamovi\'  c was partially supported by  Croatian Science Foundation under the project IP-2022-10-9006  and by the project “Implementation of cutting-edge research and its application as part of the Scientific Center of
Excellence for Quantum and Complex Systems, and Representations of Lie Algebras“, Grant No. PK.1.1.10.0004, co-financed by the European Union through the
European Regional Development Fund - Competitiveness and Cohesion Programme 2021-2027. X. Lin was partially supported by China NSF grants 12522104, 12171371. J. Yang was partially supported by China NSF grant 12371030.

D.A. and J.Y. want to thank Simons Center for Geometry and Physics at Stony Brook University and the organizers for hosting the program: Supersymmetric Quantum Field Theories, Vertex Operator Algebras, and Geometry held from March 17-April 18 in 2025, during which we made some essential progress on this project.

Part of the research reported in this paper was conducted during D.A.’s visit to the School of Mathematics and Statistics at Central China Normal University and to Wuhan University in November 2025. He is grateful to both universities for their warm hospitality. 

We also thank Sven M$\ddot{\rm o}$ller for bringing the reference \cite{GLM} to our attention. J.Y. also want to thank Thomas Creutzig for asking the questions about the existence of $M(1)^+$ tensor category, and thank Simon Lentner for some discussions related to Section 3.7.

\end{document}